\documentclass[11pt]{article}
\usepackage{hyperref}
\usepackage{amsmath,amssymb, amsfonts, amsthm, latexsym, titlesec }
\usepackage{bbm,xcolor,graphicx,mathabx}
\usepackage[margin=1in]{geometry}

\usepackage[titletoc,title]{appendix}
\usepackage{etoolbox}
\patchcmd{\appendices}{\quad}{: }{}{}

\usepackage{enumerate}
\usepackage{caption,subcaption}
\usepackage{booktabs,tabularx}

\newcommand{\ra}[1]{\renewcommand{\arraystretch}{#1}}

\usepackage{graphicx}
\usepackage{placeins}

\usepackage[linesnumbered,ruled]{algorithm2e}% http://ctan.org/pkg/algorithm2e
\DontPrintSemicolon

\newtheorem{thm}{Theorem}

\newtheorem{lem}{Lemma}[section]
\newtheorem{lemma}[lem]{Lemma}

\newtheorem{propos}[lem]{Proposition}
\newtheorem{observation}[lem]{Observation}

\newtheorem{claim}[lem]{Claim}

\newtheorem{conjecture}{Conjecture}

\newcommand{\colordependent}[2]{#2} %color

\newcommand{\putat}[3]{\begin{picture}(0,0)(0,0)\put(#1,#2){#3}\end{picture}}

\newcommand{\supp}{supp}
\newcommand{\F}{\mathcal{F}}

\newcommand{\D}{\mathcal{D}}

\newcommand{\N}{\mathbb{N}}
\newcommand{\Z}{\mathbb{Z}}
\newcommand{\R}{\mathbb{R}}
\newcommand{\E}{\mathbb{E}}
\newcommand{\Prob}{\mathbb{P}}
\renewcommand{\Pr}{\mathbb{P}}
\DeclareMathOperator{\even}{even}
\DeclareMathOperator{\odd}{odd}
\DeclareMathOperator{\sgn}{sgn}

\DeclareMathOperator{\U}{Unif}

\DeclareMathOperator{\Disc}{Dis}

\newcommand{\ind}{\mathbbm{1}}
\newcommand{\cH}{{\cal H}}
\newcommand{\cR}{{\cal R}}

\newcommand{\cO}{{\cal O}}

\title{The power of online thinning in reducing discrepancy}
\author{
Raaz Dwivedi\thanks{University of California, Berkeley, email:\texttt{raaz.rsk@berkeley.edu}.}
\and Ohad N. Feldheim\thanks{Hebrew University of Jerusalem 1613091, email:\texttt{Ohad.Feldheim@mail.huji.ac.il}, research conducted in Stanford University and supported in part by NSF grant DMS 1613091.}
\and Ori Gurel-Gurevich\thanks{Hebrew University of Jerusalem, email:\texttt{Ori.Gurel-Gurevich@mail.huji.ac.il}, research was supported by the Israel Science Foundation (grant No. 1707/16).}
\and Aaditya Ramdas\thanks{University of California, Berkeley, email:\texttt{aramdas@berkeley.edu}.}
}
\begin{document}

\maketitle
\begin{abstract}
Consider an infinite sequence of independent, uniformly chosen points from
$[0,1]^d$. After looking at each point in the sequence, an overseer is allowed to either keep it or reject it, and this choice may depend on the
locations of all previously kept points. However, the overseer must keep at least one of every two consecutive points.
We call a sequence generated in this fashion a \emph{two-thinning} sequence.
Here, the purpose of the overseer is to control the discrepancy of the empirical distribution of points,
that is, after selecting $n$ points, to reduce the maximal deviation of the number of points inside any
axis-parallel hyper-rectangle of volume $A$ from $nA$.
Our main result is an explicit low complexity two-thinning strategy which guarantees discrepancy of $O(\log^{2d+1} n)$ for all $n$ with high probability
(compare with $\Theta(\sqrt{n\log\log n})$ without thinning). The case $d=1$ of this result answers a question of Benjamini.

We also extend the construction to achieve the same
asymptotic bound for ($1+\beta$)-thinning, a set-up in which rejecting is only allowed with probability $\beta$ independently for each point.
In addition, we suggest an improved and simplified strategy which we conjecture to guarantee discrepancy of $O(\log^{d+1} n)$ (compare with $\theta(\log^d n)$, the best known construction of a low discrepancy sequence). Finally, we provide theoretical and empirical evidence for our conjecture, and provide simulations supporting the viability of our construction for applications.
\end{abstract}

\section{Introduction}
Let $(\Omega,\mathcal F, \mu)$ be a probability space and let $\cR$ be a class of subsets of $\Omega$.
The $\cR$-\emph{discrepancy} of $S$, a subset of $\Omega$ of size $n$, with respect to $\mu$ is defined as
\[
\Disc_\cR (S) := \sup_{R \in \cR} \Big| \left|S\cap R\right| - n \mu(R) \Big|.
\]
Let $X =\{X_n\}_{n\in \N}$ be a sequence of elements in $\Omega$, and write $X^n=\{X_i\}_{i\in[n]}$. The discrepancy of $X$ is defined as the sequence of discrepancies $\{\Disc_{\cR}(X^n)\}_{n\in\N}$.

Throughout we consider only with $\cR$-discrepancy with respect to Lebesgue measure on $\Omega=[0,1]^d$, where $\cR=\{\bigotimes_{i=1}^d [a_i, b_i) \subseteq [0,1)^d : 0\le a_i<b_i\le 1\}$ are the axis aligned hyper-rectangles. For brevity we call this simply \emph{discrepancy}, and denote $\Disc(X^n)=\Disc_\cR (X^n)$.

The best known upper bound for the discrepancy of $X$ is $\Disc(X^n) = O(\log^d n)$ and several lattice related constructions are known (see, e.g. \cite{Bilyk-chapter}).
However, in many applications only restricted control over the locations of the points $X_n$ is available so that an optimal discrepancy sequence cannot be used. The most extreme case is the \emph{Monte-Carlo} setting, where the points  are independent samples of the uniform distribution over $[0,1)^d$. In this case classical results in probability theory imply that $\Disc(X^n)=O(\sqrt {n\log\log n})$ and that this estimate is tight. Due to the significant gap between the optimal discrepancy obtainable by an infinite sequence and the discrepancy of a sequence of independent samples it has been desirable to look for variations on the Monte-Carlo setting which obtain lower discrepancy by allowing an overseer mild control over the sequence $X$. The most well known result in this line of investigation is the ``power of two-choices'' paper, by Azar-Broder-Karlin-Upfal \cite{ABKU}, who show that in the setting of $\Omega=[N]$, uniform $\mu$ and $\cR=\{\{n\}\ :\ n\in [N]\}$,
by allowing the overseer to choose $X_n$ among two i.i.d. $\mu$-distributed samples it is possible to obtain an exponential improvement in the discrepancy.

In this work we investigate a related, weaker sense of control. Consider an infinite sequence $U_1^{\infty} := \{U_n\}_{n\in\N}$ of i.i.d. uniform random variables on $[0,1)^d$. These points are shown to an overseer one by one, who may depend on his past choices in deciding whether to keep each point or reject it. However his control is restricted by the constraint of keeping at least one of every two consecutive points. We call a strategy executed by the overseer in producing such a sequence a \emph{two-thinning strategy}. We also consider an even weaker setting, in which, in addition to the restriction of a two-thinning, each point has independent probability $\beta$ to be rejectable and otherwise it must be kept. Inspired by the work of
Peres-Talwar-Wieder \cite{PTW} on $(1+\beta)$-choice, we call this setting \emph{$(1+\beta)$-thinning}. More precise definitions of the above terminology are provided in Section~\ref{sec: prelim}.

Our main result is an explicit $(1+\beta)$-thinning strategy on $\U [0,1)^d$, which we call \emph{Haar} strategy which satisfies the following.
\begin{thm}\label{thm:main}
The Haar $(1+\beta)$-thinning strategy yields a sequence $Z$ which almost surely satisfies
$$
\limsup_{n\to\infty} \frac{\Disc(Z^n)}{\log^{2d+1}(n)} \le \frac{100 (d^2+1)}{\beta} \ .
$$

\end{thm}
This result is obtained as an immediate consequence of the following more detailed theorem.
\begin{thm}\label{thm:main1}
The Haar $(1+\beta)$-thinning strategy yields a sequence $Z$ which
for all $n\in \N$ and $\Delta>0$ satisfies
 %\arcomment{might have to be $n \geq 3$},
{
%\small
$$
\Prob\Big(\Disc(Z^n) \ge \beta^{-1}\log^{2d}(n)(\Delta+1000+100d^2\log n) \Big)\le\beta^{-2}e^{-\frac{\Delta}{50}}.
$$
\normalsize
}
Moreover, in order to apply this strategy the overseer requires $\cO(n \log^{d}n)$ memory and $\cO(n \log^d n)$ computations to produce the first $n$ samples.
\end{thm}

%
%Here we address the problem of designing a strategy which chooses an infinite subsequence $Z_1^{\infty}$ of $U_1^{\infty}$, in an online manner, discarding on average an arbitrarily small constant proportion of any prefix, such that $Z_1^{\infty}$ achieves a lower discrepancy than $1/\sqrt{n}$.

In section~\ref{sec:Greedy Haar} we suggest a heuristic improvement of our analysis, bringing us to make the following conjecture.

\begin{conjecture}\label{conj:imp1}
The Haar $(1+\beta)$-thinning strategy yields a sequence $Z$ that almost surely satisfies
$$
\limsup_{n\to\infty} \frac{\Disc(Z^n)}{\log^{3d/2+1}(n)} <\infty.
$$%
%for any $\Delta>1$,\\[4pt]
%$$ \Prob\left(\exists n\in \N\ :\ \Disc(X^n)>\Delta c' d \log_2^{d+1}(n)\right)<\frac{C}{\beta^2}e^{-c\Delta},$$
%where $c,C, C'>0$ are absolute constants.
\end{conjecture}

In the same section we also suggest a simplified strategy with the same complexity which we call \emph{greedy-Haar} strategy, which we conjecture to provide an additional improvement over the result above.  Namely

\begin{conjecture}\label{conj:main1}
The greedy-Haar $(1+\beta)$-thinning strategy yields a sequence $Z$ that almost surely satisfies
$$
\limsup_{n\to\infty} \frac{\Disc(Z^n)}{\log^{d+1}(n)} <\infty.
$$%
%for any $\Delta>1$,\\[4pt]
%$$ \Prob\left(\exists n\in \N\ :\ \Disc(X^n)>\Delta c' d \log_2^{d+1}(n)\right)<\frac{C}{\beta^2}e^{-c\Delta},$$
%where $c,C, C'>0$ are absolute constants.
\end{conjecture}

To demonstrate that our constructions are also viable in practice as an alternative for Monte-Carlo i.i.d. sampling we dedicate Section~\ref{sec:experiments} to simulations, comparing the performance of our  strategies with classical Monte-Carlo discrepancy. Further discussion on the potential applications of our results in statistics is provided in Section~\ref{subs:application}.

\section{Related Work}

In this section we briefly survey related work on the power of two-choices and discrepancy theory and present possible applications of our work to numerical integration and statistics.

\subsection{Two-choices and weaker forms of choices}

\emph{The power of two choices} is a phenomenon discovered and popularized by Azar, Broder, Karlin and Upfal \cite{ABKU}, who consider a setting in which the underlying space is the discrete set $[M]=\{1,\dots,M\}$ and the discrepancy is measured with respect to $\cR=\{\{m\}\ :\ m\in [M]\}$. Thinking of the points $Z_1,\dots Z_N$ as balls and of their values in $[M]$ as bins, the authors considered a process where at each step a ball is assigned to the least occupied among two bins chosen uniformly and independently. They show that when $N=\Theta(M)$ this yields with high probability a discrepancy of $\Disc_\cR(Z^N)= (1+o(1))\left(\log \log N/\log 2\right)$ (compare with $\Disc_\cR(U^N)= (1+o(1))\left(\log N/\log \log N\right)$ when $U_i$ are i.i.d. uniform).
When $N\gg M$ their results imply that $\Pr(\Disc_\cR(Z^N)>\Delta{\log M})$ decays exponentially fast in $\Delta$, uniformly in $N$, so that the discrepancy does not grow with $N$. In addition, in this model, the load of a typical bin deviates from $N/M$ by merely a constant (compare with a typical deviation of $\Theta(\sqrt{N/M})$ and $\Disc_\cR(U^N)=\Theta(\sqrt{N\log M/M})$ for $U_i$ i.i.d. uniform). It was later discovered that these results are tight up to a constant in the exponent (see e.g. \cite{PTW}). Note, however, that significantly better iterated log bounds were obtained by Berenbrink, Czumaj, Steger and V\"ocking \cite{BCSV} for the one-sided gap between the load of the most loaded bin and the average load. For a simpler proof see Talwar and Wieder \cite{TW}.

While considering applications of the power of two choices to queuing theory, Mitzenmacher, in his thesis \cite{Mitzenmacher}, suggested the following more robust setting of ``two-choices with errors''. Peres, Talwar and Wieder \cite{PTW} later formulated this process, defining the equivalent $(1+\beta)$-choice process for $\beta\in[0,1]$. In this process, with probability $\beta$ (independent of everything else) the overseer is offered two uniformly distributed independent bins and with probability $(1-\beta)$ only one such bin is offered and no choice is allowed. $(1+\beta)$-thinning processes are closely related to $(1+\beta)$-choice processes.
In fact, a two-thinning set-up is equivalent to the corresponding two-choices set-up where the overseer is oblivious to the second available bin. Extending this argument, we see that every $(1+\beta)$-thinning processes is a $(1+\beta)$-choice process (i.e., every process that could be realized by a $(1+\beta)$-thinning strategy could also be realized by a $(1+\beta)$-choice strategy). On the other hand, Proposition~\ref{prop: choice criterion} below guarantees that every $(1+\beta)$-choice process for $\beta\le \frac{1}{2}$ is a $(1+2\beta)$-thinning process.
As Theorems~\ref{thm:main} and \ref{thm:main1} are obtained for $(1+\beta)$-thinning processes with arbitrarily small $\beta$, they are also valid in the $(1+\beta)$-choice setting.

In the balls and bins setting, both $(1+\beta)$-choice processes and $(1+\beta)$-thinning processes achieve the same asymptotic discrepancy of $\Theta(\log M)$ when $N\gg M$, the same discrepancy that could be achieved by
a two-choices process (this follows from results of \cite{PTW}). On the other hand, if one measures discrepancy by the one-sided maximal load semi-norm given by
 $$\max_{i\in [M]}\#\{n\in[N]\ :\ \xi(n)=i\}-N/M,$$
 then Berenbrink, Czumaj, Steger and V\"ocking \cite{BCSV} show that two-choices process can, in fact, achieve $\Theta(\log \log M)$, while both $(1+\beta)$-choice process and
 $(1+\beta)$-thinning processes still achieve only $\Theta(\log M)$ (again by \cite{PTW}). A similar gap between two choices, $(1+\beta)$-choice and  $(1+\beta)$-thinning for $\beta<1$ exists also in the regime $N\asymp M$, and in this regime both notions obtain no significant improvement over a no-choice setting.
 Curiously, when $\beta=1$, the optimal discrepancy obtainable by two-thinning strategy is $\Theta(\sqrt{\log N/\log\log N})$ which is strictly between the discrepancy in the no-choice setting, which is $\Theta(\log N/\log\log N)$ and the optimum in the 2-choice setting which is $\Theta({\log\log N})$. This is shown in a separate recent note by the second and third author \cite{FG}.

\subsection{Interval subdivision processes}
The case $\Omega=[0,1]$ of our result relates to a long line of investigation of so called \emph{interval subdivision processes}.

An \emph{interval subdivision process} is a sequence of points $(X_i)_{i=1}^\infty$ where $X_i\in [0,1]$.
The intervals of the process at the $n$-th step are the gaps between adjacent points in $(X_i)_{i=1}^n$,
while the \emph{empirical measure} at that step is defined as $\frac1{n}\sum_{i=0}^{n-1} \delta_{X_i}$, where $\delta_j$ is the Dirac delta measure.
When the points are chosen independently according to the uniform distribution on $[0,1]$ we call this
the \emph{uniform interval subdivision process}. By the law of large numbers, the empirical measure of this process
converges to the uniform measure almost surely as $n$ tends to infinity.

In 1975 Kakutani \cite{Kakutani} suggested a couple of alternative models for interval subdivision which he conjectured
to be more regular then the uniform process in the sense that their empirical measures should converge to the uniform distribution more rapidly. In one of these processes, which we refer to here as the \emph{Kakutani process}, the $n$-th point is selected uniformly on the largest interval (observe that there are no ties almost surely).
Kakutani conjectured that the empirical measure of the Kakutani process converges to the uniform measure. This fact was later proved by van Zwet in \cite{Zwet} and independently by Lootgiester in \cite{Lootgiester}.
Once convergence was established it remained to recover in what sense the Kakutani process is more regular than the i.i.d. uniform subdivision.

One natural measure for regularity of the convergence of the empirical measure is the discrepancy of the sequence.
A classical result of Kolmogorov and Smirnov (communicated by Donsker \cite{Donsker}), implies that the difference between $t$ and the empirical measure of interval $[0,t]$ of the uniform interval subdivision process, normalized by a factor of $\sqrt n$ converges to the standard Brownian bridge. Hence, the discrepancy of the uniform interval subdivision process is of
order $\Theta(\sqrt n)$. However, the the interval variation discrepancy of the Kakutani process was not easy to handle, and in the 1980s other properties of the process have been studied (see \cite{Pyke}).
Analysis of the interval variation discrepancy was made possible only in 2004 when Pyke and van Zwet \cite{PZ} were able to compute the empirical process of the  Kakutani process and showed that the difference between $t$ and the empirical measure of the process on the interval $[0,t]$, normalized by a factor of $\sqrt n$, converges to a Brownian bridge with half the standard deviation. In particular, this implied that the Kakutani process achieves an improvement of merely a constant factor in the interval variation discrepancy over the uniform interval subdivision process.

Circa 2014, Benjamini (see \cite{MP} and \cite{Junge}) suggested investigating how a two-choice
variant of the uniform interval subdivision process behaves. One family of algorithms which Benjamini suggested are
\emph{local} algorithms, namely ones in which the player considers only the size of the intervals which
contain the new sampled points. Two natural examples being \emph{max-2} and \emph{furthest-2} whose respective descriptions are ``pick the point located in the larger interval''
and ``pick the point furthest from all previously chosen points''.

Following the work of Maillard and Paquette \cite{MP} who studied other properties of the \emph{max-2} process,
Junge \cite{Junge} showed that the empirical measure of the \emph{max-2} process indeed converges to the uniform measure. However, both simulations and comparison with Kakutani processes, indicate that max-2 is likely to be at most as regular as the Kakutani process, thus demonstrating discrepancy of $\Theta(\sqrt n)$. This has been the primary instigator of our present work, where we show that that even in the weaker setting of $(1+\beta)$-thinning, by adopting a \emph{global} strategy the player can obtain a near optimal interval variation discrepancy of $O(\log^3 n)$.

\subsection{Discrepancy theory} Discrepancy theory -- the study of discrete objects which imitate regularity properties of a continuous counterpart, has, in fact, originated at the study of low discrepancy sequences with respect to the uniform measure on $[0,1)^d$, the very object investigated here. Traditionally, this theory is concerned with deterministic objects, trying to obtain bounds on the lowest discrepancy possible for prefixes $X_1,\dots,X_n$ of a sequence $X= \{X_i\}_{i\in\N}$ of points in $[0,1)^d$.

In $d\ge 2$ the exact optimal asymptotic behavior of the discrepancy is unknown. There exist explicit constructions of sequences $X$ whose discrepancy is $\Disc(X^n) \leq C_d \log^dn$ while for every sequence $X$ it is known (by \cite{Bilyk}) that there exist infinitely many $n$-s such that  $\Disc(X^n) \ge c_d \log^{(d+1+\gamma_d)/2}(n)$,
where and $c_d,C_d, \gamma_d \in (0,1)$ are constants depending on dimension.
%and $\alpha_d \geq (d+1)/2 + \beta_d$, where
%and the precise value of $\gamma_d\in (0, 1)$ is not known.

The upper bound is achieved using \emph{lattice rules} or \emph{digital nets}---for example, Hammersley point sets which are based on the infinite van der Corput sequence (see, e.g., \cite{Bilyk-chapter})
achieve the upper bound.
The lower bounds were obtained by Bilyk, Lacey and Vagharshakyan \cite{Bilyk}, building upon the work of Roth~\cite{Roth}.
We remark that arguments involving Haar wavelets, which play a key role in our construction, are used to prove lower bounds in classical literature (see, e.g., Ch. 3 of the book \cite{Chazelle}).
While there seem to be no prior work involving Haar wavelets as a tool for obtaining computationally-efficient constructive upper bounds, there exist recent works  \cite{ChenSkri1}, \cite{ChenSkri2}, \cite{Skri} which use the related Walsh wavelets to control an $L_2$ notion of discrepancy (weaker than the $L_\infty$ notion we consider).

Discrepancy theory is the main motivation for Conjecture~\ref{conj:main1}, as this conjecture would establish that online thinning typically achieves discrepancy which deviates by merely $\log n$ factor from the minimal discrepancy of any infinite sequence.

\subsection{Applications}\label{subs:application}

In this section we list a few potential applications of our results.
It is important to notice that our output sequence has the desirable property that it is unbiased. This is expressed in the following claim.
\begin{claim}\label{clm: unbiased}
Let $Z$ be the output sequence of either the Haar $(1+\beta)$-thinning strategy
or the greedy-Haar $(1+\beta)$-thinning strategy. Then for any integrable $f$ we have
$$\mathbb{E}\left[\frac1n\sum_{i=1}^n f(Z_i)\right] = \int_{[0,1)^d} f(x) dx.$$
\end{claim}
We postpone the proof of this claim to Section~\ref{subs:unbiased}.
Next, we divide the description of potential applications to one-dimensional and multi-dimensional.

{\bf One dimensional applications to statistics.}
Our results could be used to obtain a new method for on-line sample thinning in statistics.
Typically, thinning is not an effective practice in statistics. However, there are settings in which it is actually beneficial.
To make this concrete lets us illustrate the application of our method through an example from botany. Consider a setting in which a researcher
wishes to assess the expectation of a parameter $Y$ - the amount of a certain bacteria on a type of wild plants. It is well known that $Y$ is strongly dependent in an unknown yet smooth way on the mass of the sampled plant, a well studied parameter which we denote by $X$.
To obtain $Y$ the researcher must harvest the plant, keep it in cold storage and run an expensive procedures,
hence it is much more costly to assess $Y$ for any particular sample than to measure $X$. The researcher now travels in the jungle and measures $X$ for different plants,
he can then either discard them or keep them for measuring $Y$. By applying our results to the percentile distribution of $X$ (which is uniform by definition), we can thin an arbitrarily low percentage of our samples on-line and obtain an empirical percentile distribution of the samples of $X$ which has discrepancy of $O(\log^3n)$ rather than $O(\sqrt{n})$ discrepancy without any thinning.
As a result the average of sampled $Y$ will suffer from less variance caused by the variance of the sampled values of $X$. Hence the researcher will be able to obtain better precision for a given cost. Notice that by Claim~\ref{clm: unbiased} this method will not create any bias in the estimate of $\E(Y)$.

Other settings in which a similar application is viable include Experimental agriculture, where an organism (a plant or an animal) is raised and the parameter $X$ could be assessed at a much earlier stage of growth in comparison with $Y$ and Monte-carlo simulations in which $Y$ is obtained from $X$ by heavy computations. To read more on the benefit of thinning for Markov chain Monte-carlo (MCMC) samplers in a similar setting, see a recent work by Owen~\cite{Owen}.

{\bf Multi-dimensional applications.}
While the law of large numbers guarantees that a sequence of $n$ independent uniform random converge to the uniform distribution, the rate of this convergence is often slower than desired for practical applications. One setting where this is the case is that of Monte-Carlo numerical integration.
In this setting one approximates an intractable continuous integral $\int f(x) dx$ by a discrete average $\frac1n\sum_{i=1}^n f(U_i)$ for uniform $U_i$.
% To elaborate, since the sequence $U_1^n$ has a discrepancy of $\Theta(1/{\sqrt n})$, it implies that the aforementioned simple Monte-Carlo sum converges at a $1/\sqrt{n}$ rate to the integral for a large class of functions $f$ of bounded variation that are commonly encountered in practice.
For any arbitrary point sequence
%\footnote{We use $\mathcal{P}$ when we do not restrict attention to the uniform sequence $U_1^\infty$ or our constructed subsequence $Z_1^\infty$.}
$\mathcal{P}$, and compact subset $\mathfrak F$ of a Banach space, Holder's inequality implies that $|\frac1n\sum_{i=1}^n f(P_i) - \int f(x) dx| \leq \sup_{f \in \mathfrak F} \frac1n\|f\| \Disc(\mathcal{P}_1^n)$ for appropriate norms $\|\cdot\|$ that measure variation of functions. This is called the ``Koksma-Hlawka'' inequality when discrepancy is measured by axis-aligned rectangles and the functions have bounded ``Hardy-Krause'' variation (bounded mixed partial derivatives). Since $\Disc(U^n)=\Theta(\sqrt{n})$, the Monte-Carlo sum converges at a $1/\sqrt{n}$ rate to the integral.

One can achieve a much better rate of convergence by replacing random i.i.d. sequences by non-i.i.d. random sequence or even by a deterministic pseudo-random sequence that has lower asymptotic discrepancy than $U_1^n$.
As a result the theory of numerical and \emph{Quasi}-Monte-Carlo (QMC) integration have found applications to several results from discrepancy theory.
For more details on bounds related to discrepancy theory and QMC, readers may refer to the books \cite{Chazelle, Dick2, HAR} and the surveys \cite{Kuo,Dick} and the references therein.

Our sampling algorithm also provides an unbiased estimate for the integral (by Claim~\ref{clm: unbiased}). While the discrepancy and complexity of the algorithm are not as good as low discrepancy methods such as digital nets, it has the benefit of working even in setting where one cannot choose the points at which the function is evaluated. Moreover the output sequence has less structure than lattices based constructions.

Finally, considering application where actual thinning is undesirable, we remark that if rather than allowing to discard every point with probability $1-\beta$, we instead weight each point with a weight of either $1$ or $1-\beta$, then our result can be shown to persist.
%
%---
%
%and is fundamentally different from most number-theoretic constructions of low discrepancy sequences
%
%%let us discuss this a bit.
%The current version of our algorithm cannot beat the performance (time, memory requirements or achieved discrepancy) of state-of-the-art low discrepancy methods like digital nets. However, the latter have been investigated for decades, and our novel algorithm is fundamentally different from most number-theoretic constructions of low discrepancy sequences, and may perhaps be improved in future work. Our work may suggest new methods to apply in settings where we do not have full control in designing the point sequence arbitrarily, but must thin an available sequence. Further, our sequence may have additional desirable properties than low discrepancy
%%let us explain more
%--- for example, in one dimension our sequence is \emph{Poissonian}, and also automatically leads to unbiased estimates of integrals, meaning that $\mathbb{E}[\frac1n\sum_{i=1}^n f(P_i)] = \int f(x) dx$ --- which may be further investigated if required in certain applications. Lastly, in many applications it is desirable to control a weighted discrepancy that treats the coordinates differently, and it may be easy to tweak our algorithm to achieve this.

\section{Preliminaries}\label{sec: prelim}
In this section we formally define \emph{$(1+\beta)$-thinning strategies}, and related notions that are useful for our proofs.
We then give a sufficient condition that describes which distributions can be realized by a single step of $(1+\beta)$-thinning, and provide a few technical lemmata required to prove Theorem~\ref{thm:main1}. Throughout we follow the convention that the notation $\log$ denotes the logarithm with base $2$.
 %Since the probability theory terminology that we use may be unfamiliar to some, we make an attempt to keep our introduction gentle and provide appropriate references when required.

%\subsection{Preliminaries}

\subsection{Thinning functions and strategies}\label{subs: thing strat}

A \emph{thinning function} is a measurable function $f:[0,1)^{d} \to [0,1]$. We think of the input of such a function as a random element in $[0,1)^d$, and of its output as the probability that we decide to keep the chosen element. Formally, given $X_1, X_2\in[0,1)^d$ and an independent $U_1\sim\U[0,1)$, we let $Z'$ be equal to $X_1$ if $U_1\le f(X_1)$ and equal to $X_2$ otherwise. We call $Z_1$ the \emph{two-thinned sample} produced by $f$.

A \emph{two-thinning strategy} is an instrument instructing the overseer how to choose a thinning function to produce $Z_n$ given $Z_1,\dots, Z_{n-1}$. Formally, such a strategy is a countable collection of measurable functions
$f_n:([0,1)^{d})^{n-1} \times [0,1)^{d} \to [0,1]$, such that for every fixed value of the first $n-1$ entries, the function on the last entry is a thinning function.

A {two-thinning} strategy is applied to produce a random \emph{two-thinning sequence} in the following way. Denote by $X$ a sequence of i.i.d. uniform random variables on $[0,1)^d$. We now inductively define $Z$ as a subsequence of $X$ produced by the strategy. To do so, we shall employ $U=\{U_n\}_{n\in\N}$ a sequence of i.i.d. $\U[0,1]$ random variables, independent from everything else, serving as an external source of randomness. Given $Z_1,\dots, Z_{n-1}$, inductively define
$$\chi_n = \ind\{U_n>f_n((Z_1,\dots,Z_{n-1}),X_{n+\sum_{i=1}^{n-1}\chi_i})\}.$$
Here, $\chi_n$ represents the decision whether to reject (1) or keep (0) in the $n$-th step so that $\sum_{i=1}^n \chi_i$ is the number of rejections made by our algorithm in the process of allocating the first $n$ balls. Using these we set $Z_n = X_{n + \sum_{i=1}^{n}\chi_i}$.
Observe that, conditioned on $Z_1,\dots,Z_{n-1}$, the variable $Z_n$ indeed has
the distribution of a two-thinning sample according to $f(\cdot)=f_n((Z_1,\dots,Z_{n-1}),\cdot)$.

\subsection{\texorpdfstring{$(1+\beta)$}{1+beta}-thinning strategy}
Given a fixed $\beta\le 1$, a thinning function $f$ satisfying $f\ge  1-\beta$ almost surely is called a \emph{$(1+\beta)$-thinning function} and a two-thinning sample of such a function is called a \emph{$(1+\beta)$-thinned sample}. A \emph{$(1+\beta)$-thinning strategy} is a two-thinning strategy which, for every given $Z_1,\dots, Z_{n-1}$, satisfies that $f(x)=f_n((Z_1,\dots,Z_{n-1}),x)$ is an $(1+\beta)$-thinning function. Observe that such a strategy rejects each sample, conditioned on the past, with probability at most $\beta$ and that the case $\beta=1$ coincides with our previous definitions.

\subsection{Distribution realization via \texorpdfstring{$(1+\beta)$}{1+beta}-thinning}
In this section we provide a sufficient condition for a distribution on $[0,1)^d$ to be realizable as a $(1+\beta)$-thinned sample.

\begin{propos}\label{prop: choice criterion}
Let $\mu$ be an absolutely continuous probability measure on $[0,1)^d$ whose density $g$ satisfies
\[
1-\frac{\beta}2\le g(x)\le 1+\frac{\beta}2.
\]
Then, $f(x)=g(x)-\frac{\beta}2$ defines a $(1+\beta)$-thinning function whose
 % and the sequence $f_1,\dots,f_n,\dots$ defines an $\eta$-retry strategy
  $(1+\beta)$-thinned sample is distributed according to $\mu$.
\end{propos}
\begin{proof}
% \subsection{Proof of Proposition~\ref{prop: choice criterion}}\label{subs: choice criterion pf}
Let $X_1,X_2\sim\U([0,1)^d)$ and $U_1\sim \U([0,1])$, independent from one another and  let $Z'$ be equal to $X_1$ if $U_1\le f(X_1)$ and to $X_2$ otherwise, so that $Z'$ is  a $(1+\beta)$-thinned sample of $f$. We compute
\begin{align*}
\Prob(Z'\in A)&=
\Prob\Big(X_1\in A,U_1\le f(X_1)\Big) + \Prob\Big(X_2\in A,U_1 > f(X_1)\Big)\\
& =\int_A \left(g(z)-\frac{\beta}{2}\right) dz +  \int_{[0,1)^d} \left(1-g(z)+\frac{\beta}2\right)  dz\cdot |A| \\
&= \left(\mu(A)-\frac{\beta}2 |A|\right)+\left(1-1+\frac{\beta}2\right)\mathcal |A|=\mu(A),
\end{align*}
where $|A|$ is the Lebesgue measure of $A$. The proposition follows.
\end{proof}
Proposition~\ref{prop: choice criterion} is pivotal in the indirect constructions of this paper. Rather than describing
thinning functions we shall describe a discrete time stochastic process on $[0,1)^d$ whose $n$-th entry represents the location of the $n$-th ball. We then show that almost surely at every step the distribution of the next ball is realizable as a $(1+\beta)$-thinned sample for some easily computable $f$.
%In our proof we shall use Proposition~\ref{prop: choice criterion} to emulate a process drifting towards $0$ using the power of one-retry.

\subsection{Processes defined via a conditional density function}
%{\bf Marked point processes.}

Let $(\Omega,\F)$ be the measurable space on $([0,1)^d)^{\N}$ with the sigma field generated by the cylindrical Borel topology.
We call an $(\Omega,\F)$-measurable random variable \emph{a discrete time process} on $[0,1)^d$. Each process $Z=\{Z_n\}_{n\in\N}$ of this sort is associated with a counting process
$\nu=\{\nu_n\}_{n\in\N}$ defined by $\nu_n = \sum_{i=1}^n\delta_{Z_i}$ where $
\delta_x$ is a dirac delta measure at $x$. We will only concern ourselves with processes whose counting measure $\nu$ is Markovian. That is,
$$ \nu_n\ |\ \nu_1,\dots, \nu_{n-1}\overset{d}{=}
\nu_n\ |\ \nu_{n-1}.$$
These are processes satisfying that the distribution of $Z_n$ depends only on the overall locations of the previous $n-1$ balls, and not on their order.

One way to construct a exchangeable discrete time process on $[0,1)^d$ is via a \emph{conditional density function}, which we define as
sequence of measurable functions
$\lambda_n(\nu)$, each of which takes as input a counting measure of $n$ elements in $[0,1)^d$ and produces a density function $\lambda_n$ of a probability measure on $[0,1)^d.$
Given such $\lambda_n$, we write
\[
\lambda^A_n(\nu)=\int_A\lambda_n(\nu)(x)dx,
\]
 for every measurable $A\subset [0,1)^d$.%

Given such a {conditional density function} $\lambda$, we define the process $Z$ associated with it by
\begin{align}
\qquad\Pr\Big(Z_{n}\in A\ \Big|\ \{Z_i\}_{i< n}\Big) &= \lambda^{A}_{n}\left(\nu_{n-1}\right). \label{eq: cond-int-rate}
\end{align}

We call $Z$ the process associated with counting measure $\nu$ and conditional density $\lambda$.

\subsection{Balancing pairs}
%Here we are mainly concerned with point processes whose empirical distribution by time $t$ is concentrated around the uniform measure,
%with total measure $t$. We call these \emph{$\theta$-standardizing} SCPPs.
%We will employ two types of linear functionals of space-time point
%processes, both of which are drifting towards zero.
%\arcomment{Ohad, please check.}\ofcomment{Corrected to be time-scalable. Because I have multiplied the $\theta$ by $\kappa$ rather than $\kappa/2$ for standardizing some slight modifications will be needed in application (taking factor 2 into account). didn't have time to do these. }
Let $Z$ be a process on $[0,1)^d$ associated with counting measure $\nu$ and conditional density $\lambda$.
Given two disjoint sets $A,B\subset [0,1)^d$ satisfying $|A|=|B|=\kappa$,
we say that $Z$ is \emph{$\theta$-balancing} with respect to the pair $\{A,B\}$ from time $s\in \N$ if almost surely,
$\kappa\le \lambda_n^{A} +\lambda_n^{B} \leq 3\kappa$ and
\begin{equation}\label{eq: balancing property}
\begin{split}
\lambda^A_n\ge \lambda^B_n + \theta\kappa& \quad \text{ if }\nu_n(A)< \nu_n(B) ,\\
\lambda^B_n\ge \lambda^A_n + \theta\kappa& \quad \text{ if }\nu_n(B)< \nu_n(A) ,
\end{split}
\end{equation}
for all $n\ge s$.

%We will later construct a point process which could be controlled using both standardizing and balancing processes. In our application the function $w$ will always be chosen to be a Haar function (see Section~\ref{sec: haar} below for the definition of such funcitons) .
We now turn to show a key concentration property of balancing pairs. Assume that $Z$ is $\theta$-balancing with respect to $\{A,B\}$ from time $s$ and let $n\ge s$. By Equation~\eqref{eq: cond-int-rate} we have,
\begin{align}
\Pr\Big(Z_n\in A\ \Big|\ \{Z_i\}_{i<n},\ Z_n\in A\cup B, \nu_n(B) < \nu_n(A) \Big) &= \frac{\lambda^{A}_{n}\!\left(\nu_{n-1}\right)}
{\lambda^{A}_{n}\!\left(\nu_{n-1}\right)+\lambda^{B}_{n}\!\left(\nu_{n-1}\right)} \notag \\ &
\le \frac{\frac12\Big(\lambda^{A}_{n}\!\left(\nu_{n-1}\right)+\lambda^{B}_{n}\!
\left(\nu_{n-1}\right)-\theta\kappa\Big)}
{\lambda^{A}_{n}\!\left(\nu_{n-1}\right)+\lambda^{ B}_{n}\!\left(\nu_{n-1}\right)}\notag\\
&=
\frac12 - \frac12\frac{\theta \kappa}
{\lambda^{A}_{n}\!\left(\nu_{n-1}\right)+\lambda^{ B}_{n}\!\left(\nu_{n-1}\right)}
\leq \frac12 - \frac{\theta}{6}. \label{eq: balancing major property}
\end{align}%

\subsection{Concentration bounds for balancing processes}
The following lemma shows that being $\theta$-balancing with respect to a pair $\{A, B\}$ implies exponential concentration of the difference between the number of balls in $A$ and $B$.

\begin{lemma}\label{lem: balancing}
Let $s\in \N$, $ 0 < \theta < 1$ and $A,B\subset [0,1)^d$ be disjoint.
If $Z$ is a process on $[0,1)^d$ which is $\theta$-balancing with respect to $\{A,B\}$ from time $s$ and
satisfies $\E\Big(\exp\big(\theta \frac{|\nu_s(A)-\nu_s(B)|}{2}\big)\Big)\le \frac{150}{\theta^2}$,
then for all $n\ge s$ we have
$$\E\Bigg(\exp\Big(\theta \frac{|\nu_n(A)-\nu_n(B)|}{2}\Big)\Bigg)\le \frac{150}{\theta^2}.$$
\end{lemma}
%
%We will also need the following estimate, used to control the probability that an MC with bounded density will have a large number of balls in a short time interval.
%

To show Lemma~\ref{lem: balancing} we shall employ the following super-martingale type criterion.

\begin{lemma}\label{lem : martingale}
Let $(M_k)_{k\ge0}$ be random variables taking values in $\R_+$ which satisfy
\[\E ( M_k\ |\ F_{k-1}) \le \alpha M_{k-1} + \beta\]
for some $0<\alpha<1$, $\beta>0$, where $F_{n}=\sigma\big((M_k)_{0\le k\le n}\big)$.
Then
\[ \E(M_k)\le \left(1-\alpha^k\right)\frac{\beta}{1-\alpha} + \alpha^k \E(M_0) \quad\quad \text{for all $k$.}\]
\end{lemma}
\begin{proof}
\[\E(M_k)=\E(\E( M_k\ |\ F_{k-1}))\le \alpha  \E(M_{k-1}) + \beta.\]
Using induction over $k$ the lemma follows.
\end{proof}
We are now ready to prove Lemma~\ref{lem: balancing}.

\begin{proof}[Proof of Lemma~\ref{lem: balancing}]

%Without loss of generality, for the proof one may consider $\kappa = 1$. Indeed, given a process with $\kappa = \gamma \neq 1$ one may redefine time as $t' = t/\gamma$ and check that the new process $Y_{t'}$ is $\theta$-balancing with $\kappa=1$ and apply the theorem to this process (the key observation is that the bound is uniform with time and the RHS of the bound is independent of time, so any result that holds for a process also immediately holds for a sped-up or slowed-down process).
%The proof is very similar to the proof of Lemma~\ref{lem: concentration} although with a somewhat more careful treatment of
%the starting value.
Let $A, B\subset [0,1)^d$ be disjoint, $0<\theta<1$ and $s\in \N$, and assume that
$Z$ is a process on $[0,1)^d$ which is $\theta$-balancing with respect to $\{A,B\}$ from time $s$, which satisfies
\begin{equation}
\E\Bigg(\exp\Big(\theta \frac{|\nu_s(A)-\nu_s(B)|}{2}\Big)\Bigg)\le \frac{150}{\theta^2}
\label{eq: assumpt}
\end{equation}
%Denote by $\kappa$ the parameter in the definition of a $\theta$-balancing pair as in \eqref{eq: balancing property}.
Writing $\mathcal F_{n}=\sigma\big((Z_k)_{0\le k\le n}\big)$, we observe that for any $b>0$,
\begin{align*}
\E\Bigg(\exp\Big(\theta \frac{|\nu_n(A)-\nu_n(B)|}{2}\Big)&\ \Big|\ \exp\Big(\theta \frac{|\nu_{n-1}(A)-\nu_{n-1}(B)|}{2}\Big)
=e^{\frac{b \cdot \theta}2}, \mathcal F_{n-1} \Bigg)
\\&\le e^{\frac{b \cdot \theta}2} \Bigg(1-\kappa+\kappa\left(\frac{1-\theta/3}2e^{\frac{\theta}2}
+\frac{1+\theta/3}2e^{-\frac{\theta}2}\right)\Bigg)
\\ &= e^{\frac{b \cdot \theta}2}\Bigg(1-\kappa+\kappa\left(\cosh \frac{\theta}{2}-\frac{\theta}3\sinh \frac{\theta}{2}\right)\Bigg)
\\ &\le e^{\frac{b \cdot \theta}2} \Bigg(1-\kappa(1-e^{-\frac{\theta^2}{24}})\Bigg)
\end{align*}
Where the first inequality follows from \eqref{eq: balancing major property} and the last inequality uses a Taylor expansion of $\cosh(x)$ and $\sinh(x)$.

For $b=0$, we have
\begin{align*}
\E\Bigg(\exp\Big(\theta \frac{|\nu_n(A)-\nu_n(B)|}{2}\Big)\ \Big|\ &\exp\Big(\theta \frac{|\nu_{n-1}(A)-\nu_{n-1}(B)|}{2}\Big)=1, F_{n-1} \Bigg)
\\&\le 1-3\kappa+ 3\kappa e^{\frac{\theta}2}
\\&\le 1-\kappa(1-e^{-\theta^2/24})+3\kappa e^{\frac{\theta}2}.
\end{align*}

Using the above bounds, we obtain

\begin{align*}
\E\Bigg(\exp\Big(\theta \frac{|\nu_n(A)-\nu_n(B)|}{2}\Big)&\ \Big|\ \exp\Big(\theta \frac{|\nu_{n-1}(A)-\nu_{n-1}(B)|}{2}\Big), F_{n-1}\Bigg)\\
&\le \alpha \Bigg(\exp\Big(\theta \frac{|\nu_{n-1}(A)-\nu_{n-1}(B)|}{2}\Big)\Bigg) + \beta
\end{align*}
with $\alpha=1-\kappa(1-e^{-\theta^2/24})$ and $\beta=3\kappa e^{\frac{\theta}2}$.
Taking $M_k=\exp\Big(\theta \frac{|\nu_k(A)-\nu_k(B)|}{2}\Big)$ and observing that for all $n$ we have $\sigma\big((M_k)_{0\le k\le n}\big)\subseteq F_{n}$, we apply Lemma~\ref{lem : martingale} together with \eqref{eq: assumpt} to get
\begin{align*}
\E\Bigg(\exp\Big(\theta \frac{|\nu_n(A)-\nu_n(B)|}{2}\Big)\Bigg)\!<(1-\alpha^n)\frac{3\kappa e^{\frac{\theta}2} }{\kappa(1-e^{-\theta^2/24})}+\alpha^n\E\Big(\exp\Big(\theta \frac{|\nu_s(A)-\nu_s(B)|}{2}\Big)\Big)\!
<\!\frac{150}{\theta^2}.
\end{align*}
\end{proof}

 We also make the following observation.

\begin{observation}\label{obs: avg multiple}
Let $A^1,\dots, A^k$ be a collection of random variables such that for all $i\in[k]$ we have
$\E(\exp(cA^i))<C$ for some constants $c,C$.
Then for any non-negative $a_1,\dots,a_k$ such that $\sum_{i=1}^k a_i \leq 1$, we have
\[\E\left(\exp\Big({c\sum_{i=1}^k a_i A^i}\Big)\right)<C.\]
\end{observation}
\begin{proof}
This is an immediate consequence of Jensen's inequality and the convexity of the exponential function.
\end{proof}

Finally, we require the following estimate.

\begin{observation}\label{obs:too dense}
Let $Z$ be an process on $[0,1)^d$, associated with counting measure $\nu$ and conditional density $\lambda$ with $\lambda_n(x)<2$ for all $n,x$.
Let $\D \subseteq [0,1)^d$ be a measurable set.  Then, for any $s\in \N, 0< \alpha<1 $, we have
\[
\E(e^{\alpha(\nu_s(\D)-\nu_0(\D))})\le e^{4s|\mathcal D|\alpha}
\]
\end{observation}
\begin{proof}
%\big(1+2|\mathcal D|(e^\alpha-1)\big)^{n_1-n_0}\le e^{2(n_1-n_0)|\mathcal D|(e^\alpha-1)}
%Observe that for all $n\in [n_0,n_1]$ we have $\lambda^{\D}_n\le 2$ so that $\nu_{n_1}(\D)-\nu_{n_0}(\D)$ is stochastically dominated by a $(2|\mathcal D|,n_1-n_0)$-Binomial distribution. Hence
If $|\mathcal D|>\frac12$, the inequality is straightforward. Otherwise, $\E(e^{\alpha(\nu_{s}-\nu_{0})\mathcal D})$ is bounded from above by the moment generating function of Binomial distribution with parameters $s$ and $2|\mathcal D|$, which is $(1+2|\mathcal D|(e^\alpha-1))^s$. Using the fact that $1+x<e^x<1+2x$ for all $x\in [0,1]$, we bound this by $e^{2s|\mathcal D|(e^\alpha-1)}\le e^{4s|\mathcal D|\alpha}$.
\end{proof}

%
%\noindent
%Lemma~\ref{lem: balancing} bears some similarity to Lemma~\cite[Lemma~3.3]{FG}.
%Nonetheless, both the underlying space and some aspects of the results are different and so we provide the proof in Appendix~\ref{app:A}. Observation~\ref{obs: avg multiple} relates the sub-exponential tail of individual random variables, to
%a sub-exponential tail of their weighted average---
%its proof is omitted due to its simplicity, following from Jensen's inequality and the convexity of the exponential function.
%The observation is quite useful, but its pro.

\section{Haar Functions}\label{sec: haar}
 A \emph{diadic interval}
 is an interval of the form
 $I=[a2^{-\ell},(a+1)2^{-\ell})$ for $\ell,a \in\Z$. We call $\ell$ the \emph{order of $I$} and write $\cO(I)=\ell$.

 Given a diadic interval $I$ of order $\ell$, we define $I_{\even}$ and $
 I_{\odd}$ as its left and right halves -- in particular, they are the unique diadic intervals satisfying $\cO(I_{\even})=\cO(I_{\odd})=\ell+1$, $\inf I_{\even}=\inf I$
 and $\sup I_{\odd}=\sup I$.
Each diadic interval $I$ satisfying $I_{\even} \subset [0,1)$ is associated with a \emph{Haar function} $H_I:[0,1)\to\{-1,0,1\}$ defined by
\[H_I(x)= \begin{cases}
          1  & x\in I_{\even},  \\
          -1 & x\in I_{\odd},\\
          0  & \text{otherwise,}
          \end{cases}
\]
and we define the order of $H_I$ by $\cO(H_I)=\cO(I_{\even})=\cO(I_{\odd}) =\cO(I)+1$.
It is not hard to verify that Haar functions associated with different diadic intervals are orthogonal with respect to the inner product $\langle f, g \rangle:=\int_{[0,1)^d} f(t) g(t) dt$, and that they forms an orthogonal basis for $L^2([0,1])$. This is known as the \emph{Haar wavelet basis}. Note that the functions here are not normalized so that $\langle H, H \rangle = |\supp(H)| \leq 1$. Also note that the indicator function of any diadic interval $I$ of order $\ell$ is orthogonal to all Haar functions of order greater than $\ell$.

These notion generalize naturally to $d>1$. A \emph{diadic rectangle} $R\subset\R^d$ is the cartesian product of diadic intervals $I_1\times\cdots\times I_d$. The Haar function $H_R:[0,1)^d\to\{\pm1,0\}$ associated with this rectangle is $H_R=\prod_{i=1}^d H_{I_i}$. The orders of these are given by
$\cO(R):=\sum_{i=1}^d\cO(I_i)$ and $\cO(H_R)=\sum_{i=1}^d\cO(H_{I_i})$.
%We also employ an abridged notation for the constant function $H_0 = H_{[0,2)^d}=\ind_{[0,1)^d}$ when the value of $d$ is clear from the context.
 %
%We overload the notation for the constant Haar function $H_0 := \ind_{[0,1)^d}$, but whether it is being used in one or $d$ dimensions should be clear from context. We also overload the use of rectangles as sets (like in $\nu_t(R)$) or as indicator functions $\ind_R$ (like when representing rectangles as combinations of Haar functions), and the usage should again be clear from context.

Write $\mathcal H^{h_1}_{h_0}=\big\{H_R\ :\ h_0\le \cO(R)\le h_1 \big\}$ for the set of diadic Haar
functions on $[0,1)^d$ of order between $h_0$ and $h_1$. As before, Haar functions form the orthogonal \emph{Haar wavelet basis} of $L^2([0,1]^d)$.
%
%We denote the inner-product between a measure $\nu_t$ on $[0,1)^d$ and $H$ by $\langle\nu, H \rangle:=\int_{[0,1)^d} H(t) d\nu_t$, and
%say that $\nu_t$ is \emph{orthogonal} to $H$, if
%$\langle \nu_t, H \rangle=0$.
For a Haar function we also define
$$H^+:=\{x\in[0,1)^d\ :\ H(x)=1\}\text{ and }H^-:=\{x\in[0,1)^d\ :\ H(x)=-1\},$$
so that $\left\langle\nu_t,H\right\rangle=\nu_t(H^+)-\nu_t(H^-)$.
\subsection{Writing arbitrary rectangles in terms of Haar functions}
As mentioned in the overview, our strategy maintains balance with respect to all Haar functions up to a certain granularity in order to control the discrepancy on arbitrary rectangles. To this end we first express every \emph{diadic} rectangle as a linear combination of
Haar functions. We then use this construction to use Haar functions to represent every rectangle whose corners are located on a lattice, and later to approximate any arbitrary rectangle.

\begin{propos}\label{props: one rect}
 For any diadic rectangle $R$ in $[0,1)^d$ of order $\ell$ we have
 $$\ind_R=\sum_{H\in \mathcal{H}_0^\ell} \frac{\langle \ind_R, H\rangle}{\langle H, H \rangle} H.$$
Moreover, $\sum_{H\in \mathcal{H}_0^\ell} \left|\frac{\langle \ind_R, H\rangle}{\langle H, H \rangle}\right|=1$.
\end{propos}
\begin{proof}
Since $\mathcal H^{\infty}_{0}$ is an orthogonal basis for $L^2([0,1]^d)$, it would suffice to show that $\langle \ind_R,H'\rangle=0$ for all $H'\in H^{\infty}_{\ell+1}$.

To this end let $H'$ be a Haar function of order greater than $\ell$ and denote $H=\prod_{i=1}^d H_i $ and $R=\otimes_{i=1}^d I_i$. Since $\cO(H)>\cO(R)$  there must exist $j\in [d]$ such that $\cO(H_j)> \cO(I_j)$. As noted before, this implies that
$\langle \ind_{I_j}, H_j\rangle=0$. Since $\ind_R=\prod_{i=1}^d \ind_{I_i}$
and $H=\prod_{i=1}^d H_i$ we obtain that $\langle \ind_R,H\rangle= \prod_{i=1}^d\langle \ind_{I_i},H_i\rangle =0$ as required.

To see the last part, observe that if $\langle \ind_{D_i}, H\rangle\neq0$ then either $D_i\subseteq \supp(H^-)$ or
$D_i\subseteq \supp(H^+)$. Hence  for any point $x\in D_i$ we have $\frac{\langle \ind_{D_i}, H\rangle}{\langle H, H \rangle}H(x) \ge 0$ from which the last part follows.
\end{proof}
\noindent Define a \emph{lattice rectangle} in $[0,1)^d$ of order $\ell \in \N$ to be a rectangle whose corners are on the lattice $2^{-\ell}\Z$. In the next proposition we provide a decomposition of lattice rectangles of order $\ell$ into diadic rectangles. %, each of which is a sum of Haar functions by Proposition~\ref{props: one rect}.

\begin{propos}\label{prop: many rects}
Every lattice rectangle of order $\ell\ge 1$ in $[0,1)^d$ can be written as the disjoint union of at most $(2\ell)^d$ disjoint diadic rectangles of order at most $\ell$.
\end{propos}

\begin{proof}
We begin by showing that any interval of order $\ell\ge 1$ in $d=1$ can be written as the disjoint union of at most $2\ell$ disjoint diadic intervals. We prove using induction on $\ell$. For the case $\ell=1$ the statement is straightforward. For a diadic interval $I=[a2^{-\ell},b2^{-\ell})$ with $0\le a< b \le 2^{\ell}$  we write
$$I=[a2^{-\ell},a'2^{-\ell}) \cup [a'2^{-\ell},b'2^{-\ell}) \cup [b'2^{-\ell},b2^{-\ell})$$
where $a'=a+\ind_{a \text{ is odd}}$, $b'=b-\ind_{b \text{ is odd}}$ and we interpret $[c,c)=\emptyset$. Since the middle interval is of order at most $\ell-1$, by our induction assumption, it can be written as a disjoint union of at most $2\ell-2$ diadic intervals.

For general $d$, given $R=\otimes_{j=1}^d I_j$ with $I_j=[a_j 2^{-\ell},b_j 2^{-\ell})$ this allows us to decompose each $I_j$ into disjoint diadic intervals $I_{j1},\dots,I_{jk_j}$ for $k_j\le 2\ell$. Writing $R=\bigcup_{1\le m_j\le k_j} \otimes_{j=1}^d I_{jm_j}$.
\end{proof}

\noindent Finally, we bound the error when approximating any rectangle by a pair of lattice rectangles, one of which is slightly larger and one which is slightly smaller.

\begin{propos}\label{prop: rect-comp}
Let $\ell,d\in\N$. For any rectangle $R$ contained in $[0,1)^d$ there exist
lattice rectangles $R^-,R^+$ of order at most $d\ell$ such that  $R^- \subseteq R \subseteq R^+$
and
$0 \le |R^1\setminus R^2|\le 2d 2^{-\ell}$.
\end{propos}
% \arcomment{
% \begin{propos}\label{prop: rect-comp}
% Consider an arbitrary rectangle $R^0=\otimes_{i=1}^d[x_i,y_i) \subseteq [0,1)^d$. Then
% for all $h \in \N$, there exist integers $a^1_i,b^1_i,a^2_i,b^2_i \in \{0,\dots, 2^{\lfloor h/d \rfloor}-1\}$ defining expanded and contracted lattice rectangles of order $\leq h$, defined as
% $\otimes_{i=1}^d[a^1_i 2^{-\lfloor h/d \rfloor},b^1_i 2^{-\lfloor h/d \rfloor}) = R_1 \supseteq R_0 \supseteq R_2 = \otimes_{i=1}^d[a^2_i 2^{-\lfloor h/d \rfloor},b^2_i 2^{-\lfloor h/d \rfloor})$ which satisfy
% $0 \le |R^1\setminus R^2|\le 2d 2^{-\lfloor h/d \rfloor}$.
% \end{propos}
% }

\begin{proof}
Let $R=\otimes_{i=1}^d[x_i,y_i) \subseteq [0,1)^d$,
and write $R^-=\otimes_{i=1}^d[\lceil 2^{\ell} x_i\rceil2^{-\ell},\lfloor2^{\ell} y_i\rfloor2^{-\ell})$ and
$R^+=\otimes_{i=1}^d[\lfloor 2^{\ell} x_i\rfloor2^{-\ell},\lceil2^{\ell} y_i\rceil2^{-\ell})$. Clearly
$R^- \subseteq R \subseteq R^+$.
Writing $r_i= \lfloor  2^{\ell} y_i \rfloor 2^{-\ell}-\lceil 2^{\ell} x_i \rceil2^{-\ell}$ we have $|R^-|=\prod_{i=1}^d r_i$ and $|R^+|\le \prod_{i=1}^d( r_i+2^{1-\ell})$. The Propositions follows.
\end{proof}

\section{The Haar \texorpdfstring{$(1+\beta)$}{1+beta}-thinning strategy}\label{sec:Lowdisc}
In this section we present the \emph{Haar $(1+\beta)$-thinning strategy} which guarantees asymptotically low discrepancy, and show that it satisfies Theorem~\ref{thm:main1}.

Throughout, let $h = h(n) =\lfloor \log  n  \rfloor$. This will serve
as the largest order Haar function being considered by our strategy at time $n$. We denote $W(s)=\sum_{i=1}^s{\binom{i+d-1}{d-1}}$ and let $Z$ be a process on $[0,1)^d$ associated with counting measure $\nu$ and conditional density $\lambda$, defined by
\begin{equation}\label{eq:intensity}
\lambda_n(\nu_n)(x) := 1 + \frac{\beta}{2W(h)} \sum\limits_{H\in \cH_1^h}\sgn\langle  \nu_n,-H \rangle H(x).
\end{equation}

\begin{figure}
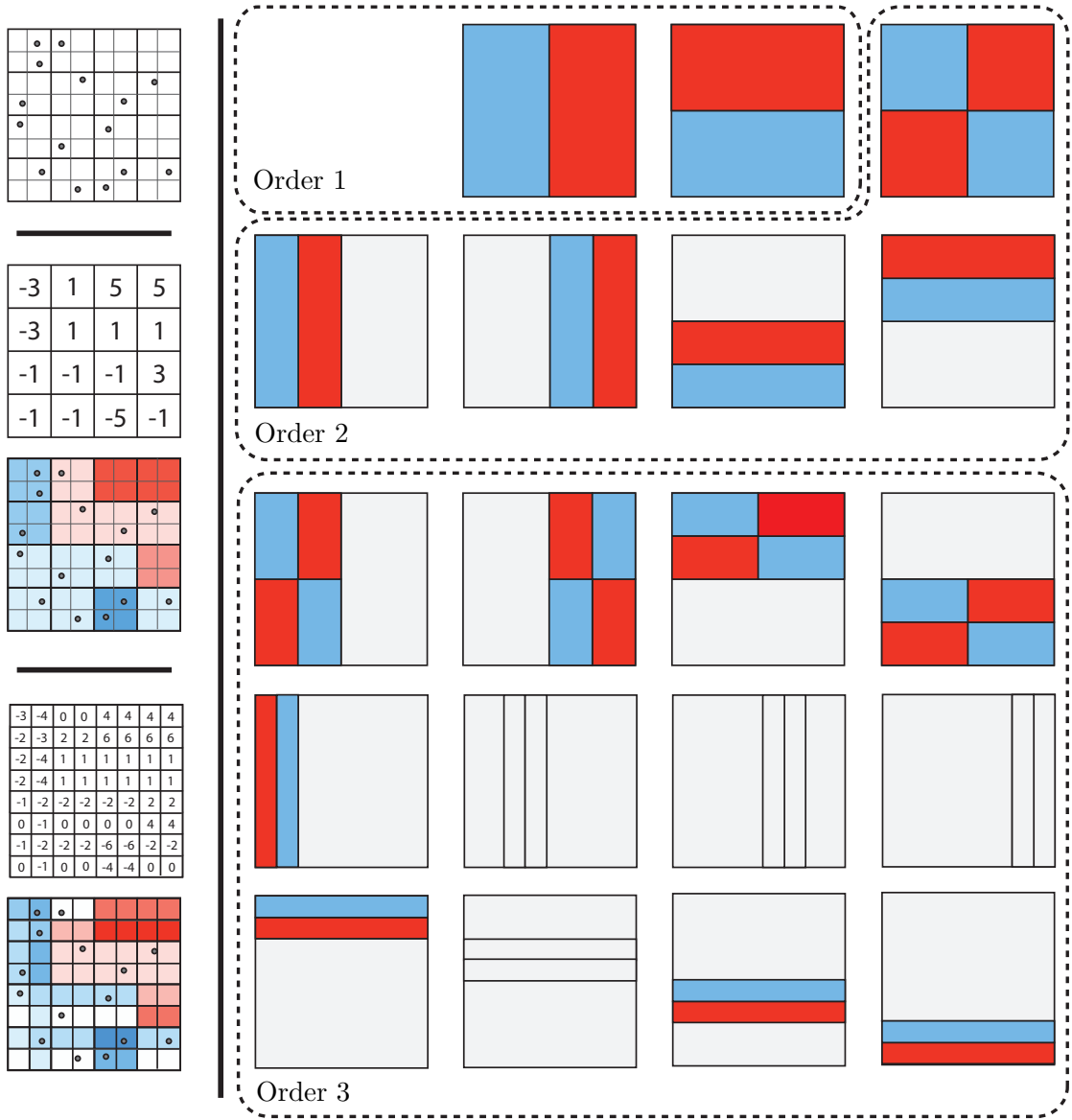

  \centering
  \begin{subfigure}{.94\textwidth}
    \centering
  \colordependent{\includegraphics[width=1.\linewidth]{diag/Haar3.pdf}}{\includegraphics[width=1.\linewidth]{diag/Haar3col.pdf}}
  \end{subfigure}%
  \putat{-325}{150}{Order 1}
  \putat{-328}{50}{Order 2}
  \putat{-331}{-210}{Order 3}
  \caption{{\bf Left, top.} $\nu_{15}$ the empirical measure of a sequence sampled according to our thinning strategy at $n=15$. {\bf Left, middle.} $\sum_{H\in \cH_1^3}\sgn\langle  \nu_n,-H \rangle H(x)$, averaged on diadic squares of side 1/4 along with a visual representation of the conditional density of $\lambda_{15}$  averaged on diadic squares of side 1/4. \colordependent{Darker color indicates higher density}{Warmer color indicates higher density}. {\bf Left, bottom.} The
  conditional density of $\lambda_{15}$ along with a visual representation.
  {\bf Right.} Haar functions of orders one to three, multiplied by $\sgn \langle  \nu_t,- H \rangle$. Gray indicates the value $0$, \colordependent{light gray}{blue} -- the value $-1$ and \colordependent{dark gray}{red} -- the value $1$. Notice that the fully grayed out functions of order three are the ones which are perfectly balanced.}
 \label{fig:intensity example}
\end{figure}

We begin by observing that
\begin{observation} \label{obs: realization}
$Z$ is a $(1+\beta)$-thinned sample of a $(1+\beta)$-thinning strategy.
\end{observation}
\begin{proof}
Observe that $\int_{[0,1)^2} H(x)dx=0$ for all $H\in \cH_1^h$ and therefore, $\int_{[0,1)^2} \lambda_n(\nu_n)(x)dx=1$. We only need to verify that the condition of Proposition~\ref{prop: choice criterion} is satisfied at every $n\in \N$, i.e. that for all $x\in[0,1]^d$ and $n\in\N$ we have $$1-\frac{\beta}2\le \lambda(x)\le 1+\frac{\beta}2,$$
which follows immediately from \eqref{eq:intensity}, and from the fact that for all $x\in [0,1)$ we have
\begin{equation}\label{eq:tilt power estimate}
\sum_{H\in \mathcal H_1^h}|H(x)|=|\{s \in \N_0^d ,0<\sum_{i=1}^d  s_i\le h\}|=\sum_{i=1}^h\binom{i+d-1}{d-1}=W(h).
\end{equation}
\end{proof}
In light of the claim we call the strategy producing $Z$ the \emph{Haar} $(1+\beta)$-thinning strategy.

Next, in Section~\ref{subs:complex} we discuss the complexity of realizing this strategy. In Section~\ref{subs:concentration} we show exponential concentration properties related to $Z$. Finally in section~\ref{subs:thmpf} we use these to prove Theorem~\ref{thm:main1}.

\subsection{Realizing the Haar thinning strategy} \label{subs:complex}

In this section we discuss the time and memory complexity required for the overseer to realize the Haar thinning strategy.
In particular we show the following.
\begin{propos}\label{prop:complexity}
In order to apply the $(1+\beta)-$Haar thinning strategy the overseer requires $\cO(n \log^{d}n)$ memory and $\cO(n \log^d n)$ computations to produce the first $n$ samples.
\end{propos}
\begin{proof}
Recall that in our set-up the overseer is given a uniformly distributed point $X_{n+\sum_{i=1}^{n-1}\chi_i}$ in $[0,1)^d$. Then, relying upon a data structure which he maintain, the overseer he must compute a threshold $\tau_n \in [0,\beta]$. Then with probability $\tau_n$ the value of $\chi_n$ is set to be 1 and otherwise it is set to be 0. In light of Proposition~\ref{prop: choice criterion}, in order to realize $Z$ we must set $\tau_n=\lambda_n(\nu_n)(x)-\frac{\beta}2$. The rest of the section discusses the complexity of computing this function.

We remark that, as the custom goes, complexity estimates are given for integer computations and ignore the increase in storage, reference and computation costs for large numbers. If these were taken into account additional poly-$\log\log n$ factors would multiply both time and memory.

As before, let $n\in\N$, recall that $h = h(n) =\lfloor \log  n  \rfloor$ and $W(s)=\sum_{i=1}^s{\binom{i+d-1}{d-1}}$, and denote by $H_I$ a Haar function corresponding to the diadic rectangle $I$.
For each function $H_I$ of order $\ell$ with $I=I_1\times\dots\times I_d$ and $\cO(I_i)=\ell_i$ so that $\ell=1+\sum \ell_i$, we call $\ell_1,\dots,\ell_d$ the \emph{shape} of $I$.
At time $n$ we maintain an array $A_n$ of gradually increasing size. $A_n$ consists of data cells corresponding to each Haar function in $\cH_1^h$. Each of these cells associated with $H_I\in \cH_1^h$ contains the present value of $\langle \nu_n, H_I\rangle$.
Observe that the size of such an array is bounded by the total number of shapes which is $W(h)=O(\log^{d} n)$, multiplied by the maximal number of elements of each shape of order $\ell\le h$ which is $2^{\ell+d}= O(n)$, giving total memory complexity of $O(n\log^{d} n)$, as required.

The arrangement of $A$ is as follows. We order the data cells first by their order $\ell$, then lexicographically by shape and then lexicographically by the point whose coordinate sum is minimal in $I$. With this arrangement we can find the cells of all Haar functions of order less than $\ell$ containing a particular point at the cost of a constant number of arithmetic operations per function.

Given this array, computing the value of $\tau_n$ takes $W(h)=O(\log^{d} n)$ operation, one for each element of the sum. Using this we can determine the value of $Z_n$. We then update $A_n$ by altering the value of all entries corresponding to Haar functions associated with rectangles containing $Z_n$. As noted in \eqref{eq:tilt power estimate}, this takes $W(h)=O(\log^{d} n)$ operations.
In addition, for each $n$ such that $h(n)>h(n-1)$ we must allocate additional  entries to $A$ for the new $\binom{h+d-1}{d-1}$ shapes of order $h(n)$. There are less than $2n$ Haar functions for each of these shapes so that this operation takes less than $2n W(h)$ operations.
We then go over all points $Z_1,\dots, Z_n$ and update the entries of $A_n$ corresponding to Haar functions associated with the new shapes at the cost of $O(nW(h))$ steps. Hence to produce the first $n$ entries and the time complexity is
$$ O(n\log^{d} n)+\sum_{s=1}^{\lfloor{\log n}\rfloor}\Big(O(2^{d+s} W(s)) + O(2^sW(s))\Big) = O(n\log^{d} n),$$
concluding the proof of the proposition.
\end{proof}

\subsection{Concentration properties of \texorpdfstring{$Z$}{Z}}\label{subs:concentration}

We begin by showing that diadic projections of $Z$ have a balancing nature. Recall that $h = \lfloor \log n \rfloor$.
\begin{propos} \label{prop: self-correcting}
For any Haar function $H$ on $[0,1)^d$ we have
\begin{equation}\label{eq: msbnb concentration}
\E\left(e^{\frac{\beta\left| \langle \nu_n, H \rangle \right|}{2W(h)}} \right)<\frac{600 W(h)^2}{\beta^2}.
\end{equation}
\end{propos}
\begin{proof}
Let $H$ be a Haar function on $[0,1)^d$ of order $\ell\in \N$. We use different arguments for times before and after $2^\ell$.
We begin by showing that $Z$ is $\frac{\beta}{ W(h)}$-balancing with respect to $\{H^+,H^-\}$ starting from time $s=2^\ell$.
We
write $\kappa:=|H^-|=|H^+|$, let $2^\ell \le  M \le n$ and observe that
\begin{align*}\lambda^{H+}_M-\lambda^{H-}_M&=\langle H,\lambda\rangle =\left\langle H,1 + \frac{\beta}{2W(\lfloor \log M \rfloor)} \sum\limits_{G\in \cH_1^{\lfloor \log M \rfloor}}G(x)\sgn\langle  \nu_M,-G \rangle  \right\rangle\\
&= \frac{2\beta\kappa\sgn\langle  \nu_M,-H \rangle }{2W(\lfloor \log M \rfloor)}.
\end{align*}
Hence the conditions of \eqref{eq: balancing property} are satisfied with $\theta=\frac{\beta}{W(h)}$. Next we show \eqref{eq: msbnb concentration}.
Indeed, for any time $M \leq 2^\ell$, we have
%$\E(e^{ \theta \frac{ |\langle \nu_{t},H \rangle|}{3}}) \leq \frac{100}{\theta^2}.$ Indeed, we have
\begin{equation}\label{eq:small-t}
%\E(e^{ \theta  \frac{|\langle\nu_{t},H \rangle|}{3}}) =
\E(e^{\frac{\beta}{2W(h)} |\langle  \nu_{M},H \rangle|}) ~\le~ \E(e^{\frac{\beta}{2W(h)} \langle  \nu_{M},|H| \rangle}) ~\stackrel{(i)}{\le}~ e^{\frac{2\beta}{W(h)}2^{\min(d,h)}}\stackrel{(ii)}{\le} 100
<
%  \frac{100}{\theta^2} =
  \frac{600 W(h)^2}{\beta^2}.
\end{equation}
Here inequality (i) follows from Observation~\ref{obs:too dense} using $M \leq 2^\ell,|\D|=|\supp(H)| \le  \min(2^{d-\ell},1)$.
To see inequality (ii) we claim that $W(h)\ge 2^{\min(h,d)-1}$ . Indeed, if $h\le d$, then
$$W(h)\ge\binom{h+d-1}{h}  \ge \frac{d^h}{h!}\ge \frac{h^h}{h!}\ge 2^{h-1},$$
while if $h> d$, then
$$W(h)\ge\binom{h+d-1}{d-1}  \ge \frac{h^{d-1}}{(d-1)!}\ge \frac{d^{d-1}}{(d-1)!}\ge 2^{d-1}.$$
%\arcomment{Is the previous calculation correct? $H$ is of level $\ell$, so its support has volume $\propto 2^{-\ell}$.}
From this we also deduce that \eqref{eq: msbnb concentration} holds in the case $n\le 2^\ell$.
By applying
Lemma~\ref{lem: balancing} with $\{H^+,H^-\}$, $s=2^\ell$ and $\theta = \frac{\beta}{W(h)}$ we get that \eqref{eq: msbnb concentration} holds also in the case $n > 2^\ell$, concluding the proof of the proposition.
\end{proof}

Next, we use this to show concentration of $\nu_n$ on low-order lattice rectangles.
\begin{propos} \label{prop: diadic rect cons}
For any $n\in \N$ and any lattice rectangle $R\subset [0,1)^d$ with of order at most $h$ we have
\begin{equation*}
\E\left(e^{\frac{\beta\left| \nu_n(R) - n |R| \right|}{2^{d+1}h^{d}W(h)}} \right)\le\frac{600 W(h)^2}{\beta^2}.
\end{equation*}
\end{propos}
\begin{proof}
Let $R$ be a lattice rectangle of order at most $h$. By Proposition~\ref{prop: many rects} there exist disjoint diadic rectangles $D_1,\dots,D_k$ of order at most $h$ such that $k\le (2h)^d$ and $R=\cup_{i=1}^k D_i$. By Proposition~\ref{props: one rect}, each $D_i$ satisfies
$$\ind_{D_i}=\sum_{H\in \mathcal{H}_0^h} \frac{\langle \ind_{D_i}, H\rangle}{\langle H, H \rangle} H,$$
with $\sum_{H\in \mathcal{H}_0^\ell} \left|\frac{\langle \ind_{D_i}, H\rangle}{\langle H, H \rangle}\right|=1$.
We observe that $\langle \ind_{D_i}, \ind_{[0,1)^d}\rangle=n |D_i|$ where $\ind_{[0,1)^d}$ is the only Haar function of order $0$.
We conclude that there exist coefficients $a_H$ for $H\in \mathcal{H}_1^h$  with $\sum_{H\in \mathcal{H}_1^h} |a_i|\le  (2h)^d$ such that
$$\ind_{R}=\sum_{H\in \mathcal{H}_1^h}a_i H+n|R|.$$

By Proposition~\ref{prop: self-correcting}, together with Observation~\ref{obs: avg multiple},
this implies that
\begin{equation}\label{eq: misuse of obs avg multiple}
\E\left(e^{\frac{\beta\left| \nu_n(R) - n |R| \right|}{2(2h)^dW(h)}} \right)\le\frac{600 W(h)^{2}}{\beta^2},
\end{equation}
as required.
\end{proof}

\subsection{Proof of Theorem~\ref{thm:main1}}\label{subs:thmpf}
Let $n\ge 4$ and observe for $n<2^d$ the theorem is straightforward as $\Disc(Z^n)\le n$ almost surely. As before denote $h=\lfloor \log n\rfloor$. We begin by bounding $W(h)$. We compute
$$W(h)\le  \frac{h(h+d-1)^{d-1}}{(d-1)!}\le \frac{2^{d-1}h^{d}}{(d-1)!}\le 25\left(\frac{h}{2}\right)^d.$$

Next, let $R$ be a lattice rectangle of order at most $h$.
By Proposition~\ref{prop: diadic rect cons} we have
$$
\E\left(e^{\frac{\beta\left| \nu_n(R) - n |R| \right|}{50h^{2d} }} \right)\le\E\left(e^{\frac{\beta\left| \nu_n(R) - n |R| \right|}{2^{d+1}h^d W(h)}} \right)\le\frac{600 W(h)^2}{\beta^2},$$
so that by Markov's inequality for all $\Delta'>0$ we have
$$\Prob\left(\Big| \nu_n(R) - n |R| \Big|\ge  \Delta' \log^{2d}n \right)\le \frac{2^{19}  \log^{2d} n}{2^{2d}\beta^2}e^{-\frac{\Delta'\beta}{50 }}.$$

Observe that there are at most $n^{2d^2}$ lattice rectangle of order at most $d\log n$. Denoting $\cR_i=\{\text{lattice rectangle }R\subset [0,1)^d\ :\ \cO(R)\le i\}$ we have
$$\Prob\left(\sup_{R \in \cR_i} \Big| \nu_n(R) - n |R| \Big|\ge  \Delta' \log^{2d}n\right)\le \frac{2^{19}  \log^{2d} n}{2^{2d}\beta^2}n^{2d^2}e^{-\frac{\Delta'\beta}{50 }}.$$

By proposition~\ref{prop: rect-comp} applied with $\ell =\lfloor\log n\rfloor$, for each rectangle $R$ in $[0,1)^d$ there exist $R^+,R^-\in \cR_{d\log n}$ such that
$R^-\subseteq R\subseteq R^+$ and $|R^+|\le|R^-|+4d/n$.
Hence
$$\nu_n(R^-)-n|R^-|-4d \le\nu_n(R^-) - n |R^+|\le\nu_n(R) - n |R|\le\nu_n(R^+) - n |R^-|\le \nu_n(R^+)-n|R^+|+4d.$$
Hence, for $\Delta'>0$ we get
$$
\Prob\left(\Disc(Z^n) \ge \Delta' \log^{2d}n +4d\right)\le \beta^{-2}\exp\left(\frac{2d^2 \log n}{\log e}+\log(\log n)+ \log(2d)+\frac{19-2d}{\log e}-\frac{\beta\Delta'}{50 }\right).
$$
Plugging in $\Delta' = \beta^{-1}(\Delta+1000+100d^2\log n)$ we obtain
$$\Prob\Big(\Disc(Z^n) \ge \beta^{-1}\log^{2d}(N)(\Delta+1000+100d^2\log n) \Big)\le\beta^{-2}e^{-\frac{\Delta}{50}},$$
as required.
\qed
%\ref{prop: rect-comp}
\subsection{Proof of Claim~\ref{clm: unbiased}}\label{subs:unbiased}

It would suffice to show the $Z$ is unbiased for $\ind_R$ where $R=\otimes_{i=1}^d I_i$ is a diadic rectangle $\ell$. To see this we show that
\begin{equation}\label{eq:bias}
\mathbb{E}\left[\frac1n\sum_{i=1}^n \ind_R(Z_i)\right] = \mathbb{E}\left[\frac1n\sum_{i=1}^n \ind_{R'}(Z_i)\right]
\end{equation}
for any $R'=\otimes_{i=1}^d I'_i$ with $\cO(I'_i)=\cO(I_i)$. This is a consequence of the diadic tree symmetry. To see this, consider the binary representation of $R$ and $R'$ in each dimension and write $D_i$ for the digits in which they disagree in dimension $i$. Let $g:[0,1)^d\to[0,1)^d$ be the measure preserving bijection which maps a point $x$ to a point $g(x)$ whose binary representation in each coordinate $i$ is flipped exactly on $D_i$.
Now couple the sequence $X_i$ and with a sequence $X'_i=g(X_i)$ and apply the same strategy to produce $\{Z_j\}_{j\in\N}$ and $\{Z'_j\}_{j\in\N}$ using the same sequence $U$ used to determine our thinning decisions as in Section~\ref{subs: thing strat}. Observe that in this case $Z'_i=g(Z_i)$ so that for all $n$ we have
$$\frac1n\sum_{i=1}^n \ind_R(Z_i) = \frac1n\sum_{i=1}^n \ind_{R'}(Z'_i),$$
and hence, as $Z'_i\overset{d}{=}Z_i$, \eqref{eq:bias} holds.
\qed

\section{The greedy-Haar strategy}\label{sec:Greedy Haar}
%Let $\psi(x)=\begin{cases}1 & 0\le x < \frac12 \\ -1 & \frac12\le x<1 \\ 0 & \text{otherwise}\end{cases}$.
%And define $\psi_{\ell,s}(x_1,\dots,x_d)=\prod_{i=1}^d s_i\psi(x_i-\ell_i)$ for $\ell,s\in \R^d$.

In this section we describe the empirically more efficient variant of our strategy called the \emph{greedy-Haar strategy}. We then provide heuristic justification for Conjectures~\ref{conj:imp1} and~\ref{conj:main1}.

Unlike the case of the Haar strategy, we describe the strategy directly by $$f_n((Z_1,\dots,Z_{n-1}),x)=\begin{cases}
1 & \sum\limits_{H\in \cH_1^h}\sgn\langle  \nu_n,-H \rangle H(x)<0\\
\frac{1}{2} & \sum\limits_{H\in \cH_1^h}\sgn\langle  \nu_n,-H \rangle H(x)=0\\
0 & \sum\limits_{H\in \cH_1^h}\sgn\langle  \nu_n,-H \rangle H(x)>0
\end{cases}$$

\noindent The name greedy-Haar corresponds to the a point of view by which each Haar function $H$ wishes to reduce $\langle  \nu_n,H \rangle$. Hence we compute $\sum_{H\in \cH_1^h}\sgn\langle  \nu_n+\ind_{x},H \rangle -\sum_{H\in \cH_1^h}\sgn\langle  \nu_n,H \rangle $ and if this quantity is positive we keep $x$, if it is negative we reject it, and if it is $0$, we break the tie by a fair coin-toss.

\subsection{Heuristic analysis}

We begin by describing the $\log^{d/2}(n)$ heuristic improvement to Theorem~\ref{thm:main},
giving rise to Conjecture~\ref{conj:imp1}. We then describe an additional heuristic $\log^{d/2}(n)$ improvement stemming from the greedy-Haar strategy which adds up to Conjecture~\ref{conj:main1}.

{\bf Improvement of the analysis (Conjecture~\ref{conj:imp1}).} We conjecture that the usage of Observation~\ref{obs: avg multiple} to obtain \eqref{eq: misuse of obs avg multiple} is not tight. In this transition we decompose
each rectangle $R$ to the sum Haar functions whose coefficients add up to at most $\log^{d}_2 n$. We then bound the rectangle's discrepancy by a triangle inequality using the bound for each individual Haar funciton. However, for a rectangle $R$ and a Haar function $h$ we have $|\langle  \ind_R,h\rangle|/ \langle h, h\rangle \le 1$, so the coefficient of each particular Haar function is at most $1$. Hence, assuming sufficient independence between the coefficients $\langle  \ind_R,h\rangle$ for different Haar functions $h$, we should expect the sum of $\langle \nu, H\rangle$ to produce a discrepancy of $\log^{d/2}_2 n$, and not $\log^{d}_2 n$.

{\bf Better concentration inequalities for the greedy-Haar strategy (Conjecture~\ref{conj:main1}).}
Let $H\in \cH^h_1$ be a particular Haar function and assume that $\langle  \nu_{n},-H \rangle>0$. Denote by $k$ the number of elements of $\cH^h_1$ whose support contains a given point.  Also recall the notation $G^-$ and $G^+$, the positive and negative domains of a Haar function $G$. We examine the probability of that a point falls in $H^+$ compared with the probability that it falls in $H^-$.
Observe that every other Haar function $G\in \cH^h_1\setminus
\{H\}$ is orthogonal to $H$ so that $(\sgn\langle  \nu_n,-G\rangle)\langle H,G\rangle)=0$.
In addition, if we approximate the signs of $\langle  \nu_{n},-G \rangle$ for $G\in \cH^h_1\setminus \{H\}$ by independent random variables , then their total value would have a binomial($\frac12,k$) distribution, so that typically on a region of size $\frac{k^{-1/2}}2 \supp(H)$ they are tied and the sign of $\langle  \nu_{n},-H \rangle$ determines whether to accept or reject.  Hence we expect the process $\langle  \nu_{n},-H \rangle$ to behave roughly like an $\Theta(1/h^{d/2})$ balancing process which would yield an improvement of $\log^{d/2}_2 n$ to the bound.

%\section{Miscellaneous Proofs}
%
%\begin{proof}[Proof outline of Fact~\ref{fact: unique}.]
%To see that such $X$ is well defined and uniquely determined, observe that by our assumption
% there exists $C$ such that $\lambda(t,\nu)(A)<C$ for every $\nu\in \mathfrak N$ and measurable $A\subset [0,1)^d$. Hence,
%$X^x_t$ is locally finite and $T_0:=0$, $T_i:=\min_t(X_t\neq X_{T_0})$ for $i\in\N$ are well defined by
%$$\Pr(T_i\ge t_0\ |\ \F_{T_{i-1}})= \ind\{ t_0\in [T_{i-1},\infty)\}\left(1-\int_{T_{i-1}}^{t_0} \lambda(t,X_{T_{i-1}})(A)e^{-\lambda(t,X_{T_{i-1}})(A)}dt\right).$$
%Writing $Y_i$ for the single point added to the process at time $T_i$, we then have for all measurable $A\subset [0,1)^d$
%$$\Pr(Y_i\in A\ |\ F_{T_{i-1}},T_i) ~\propto~ \lambda(t,X_{T_{i-1}})(A).$$
%Using this we can now easily express every $\Pr(X_t(A)\ge c)$ for all $t\in\R_+$, measurable
%$A\subset [0,1)^d$, which determine $X$. We omit the details.
%\end{proof}

\section{Empirical results}\label{sec:experiments}

In this section we provide simulation results both for the Haar and the greedy-Haar 2-thinning strategies. As evident from these simulations, the greedy-Haar strategy is significantly better than the Haar strategies, and both strategies perform somewhat better than shown by our Theorems.

We begin by showing discrepancy results, and then discuss the bias of particular rectangles. In all simulations we compare the three methods, i.i.d. samples which we refer to here as Monte-Carlo, Haar 2-thinning, and greedy-Haar 2-thinning. Unfortunately the simulations are not sufficient to determine the power of the log in the decay of the discrepancy with sufficient certainty to scientifically estimate the exponent of the log in Conjecture~\ref{conj:main1}.

\subsection{Main Simulations}

We have averaged $20$ simulated outputs of $2^{19}$ samples for each of the three strategies in one dimension. For this case, we have computed the rectangle $R$ which has maximal $|\nu_N(R)-|R||$ whenever $N=\lceil 2^k\rceil$.
Our results are summarized in Table~\ref{tab:worst} and Figure~\ref{fig:worst}.

\begin{table}
 \centering
 \resizebox{\columnwidth}{!}{
  \begin{tabular}{lllllllll}
    \hline
    \toprule
    {\bf Strategy}        & $n= 56$   & $n=2^7$     & $n=2^9$       & $n=2^{11}$      & $n=2^{13}$       & $n=2^{15}$     & $n=2^{17}$     & $n=2^{19}$     \\
    \midrule
    {\bf Monte Carlo}     & 8.6 (2.5)  & 14.3 (3.3)    & 25.3 (6.1)      & 55.8 (3.5)    & 108.6 (18.4)   & 247.1 (43.9) & 415.8 (38.4) & 835.3 (255.2)\\
    {\bf Haar}      & 8.1 (2.6)  & 11.7 (2.2)  & 28.3 (4.9)    & 43.5 (10.2)     & 97.3  (25.4)   & 128.8 (40.9) & 251.4 (59.1) & 399.8 (134.7)\\
    {\bf greedy-Haar} & 5.9 (0.7)  & 7.8  (1.4)  & 13.0 (8.4)    & 20.3 (5.5)      & 28.0  (2.9)    & 37.0  (2.9)  & 51.9  (1.9)  & 67.5 (4.0)\\
    % {\bf 2-Thinning} & 1 & 2 & 3\\
    % {\bf Greedy-Thinning} & 1 & 2 & 3\\
    \bottomrule
    \hline
  \end{tabular}
  }
  \caption{Some Values of Discrepancy for different strategies in one-dimension.
  Given are the mean (and standard error) across 20 experiments.}
  \label{tab:worst}
\end{table}

\begin{figure}
  \centering
  \begin{subfigure}{.49\textwidth}
    \centering
    \includegraphics[width=1.\linewidth]{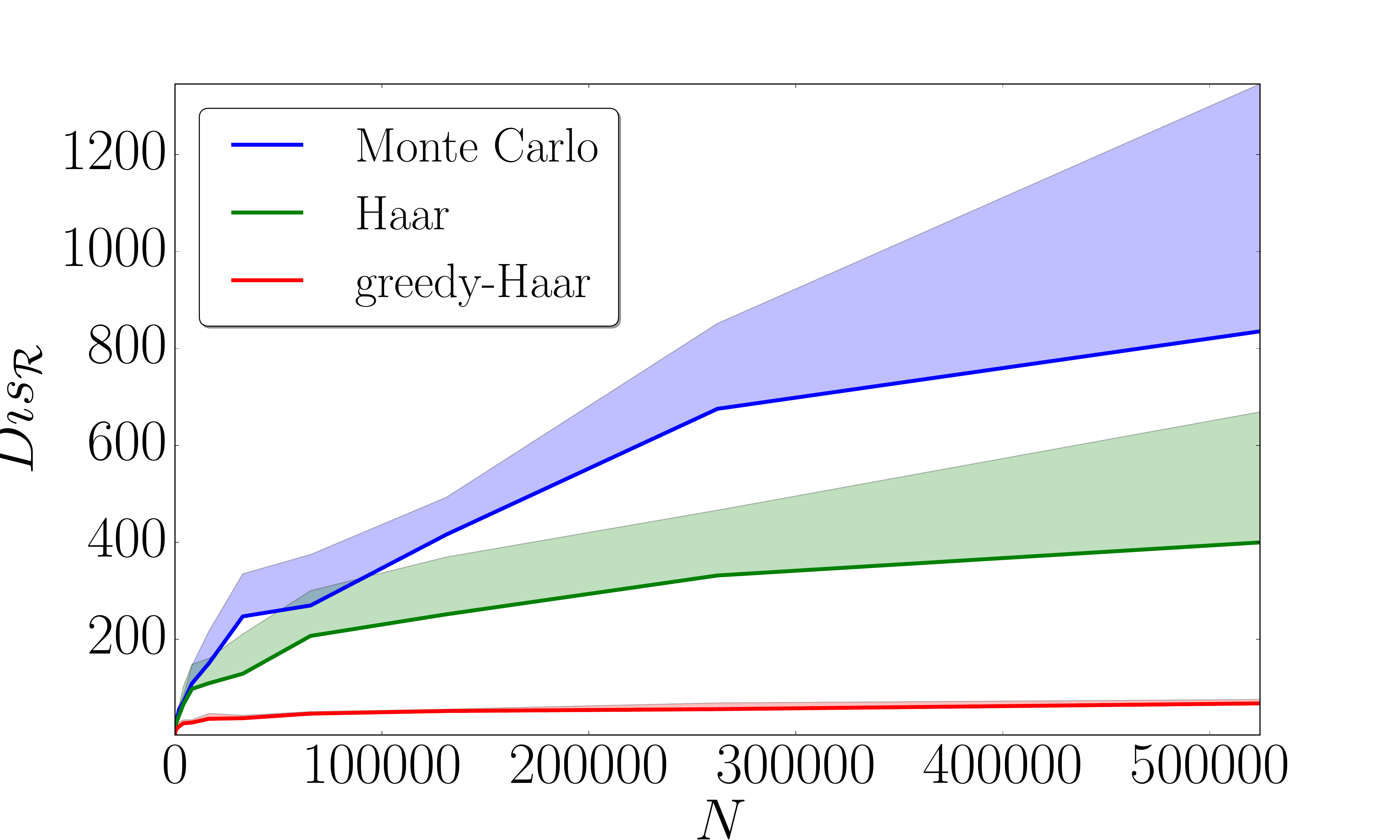}
    \caption{$d=1$, Linear Scale}
  \end{subfigure}%
    \begin{subfigure}{.49\linewidth}
    \centering
    \includegraphics[width=1.\linewidth]{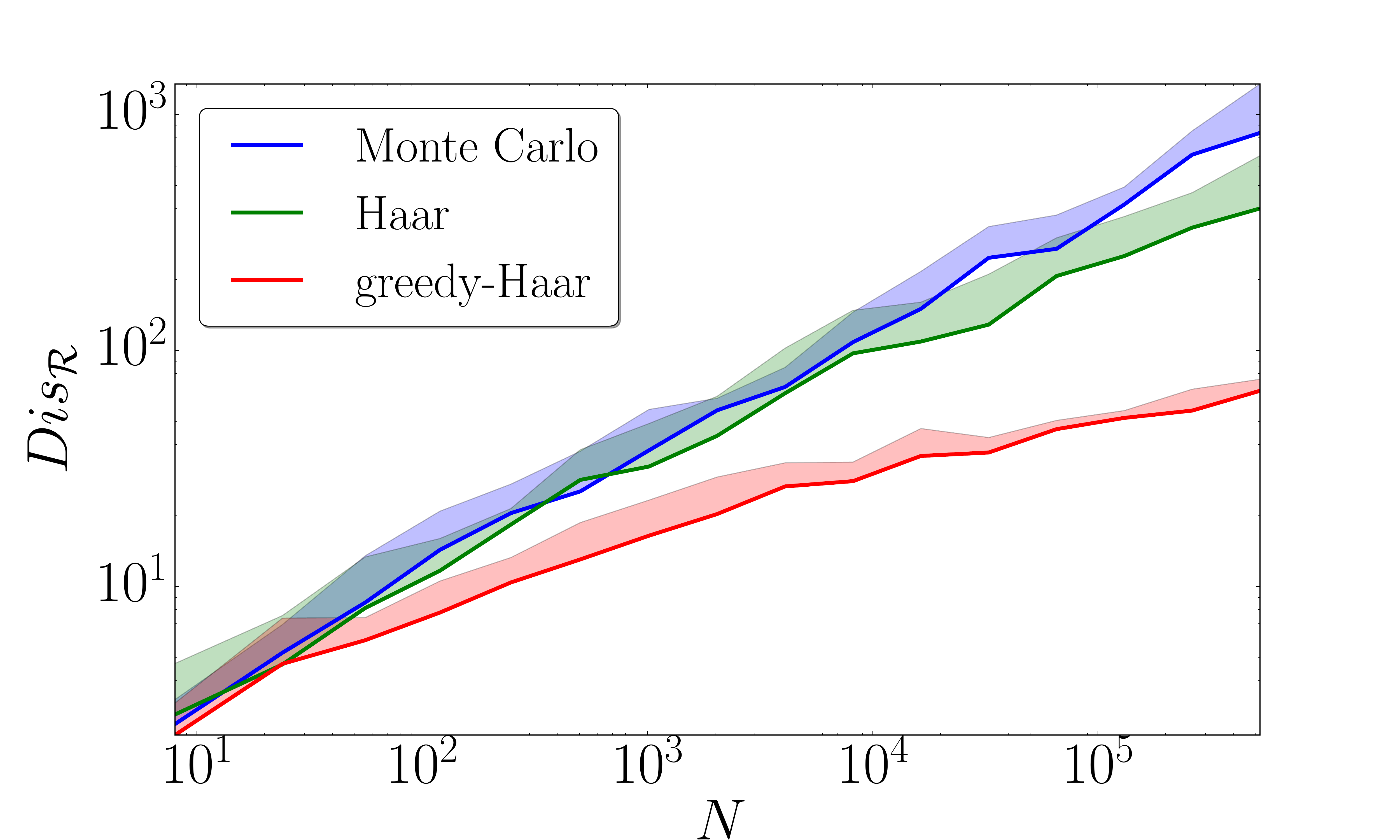}
    \caption{$d=1$, Log Scale}
  \end{subfigure}\vspace{5mm}

  \caption{Plots of discrepancy (averaged over $20$ experiments) in one dimensions for Monte Carlo (blue), Haar 2-Thinning (green) and Greedy-Thinning 2-Thinning (red) strategies. The plots are provided both in linear and log-scale.
}
 \label{fig:worst}

\end{figure}

\subsection{Other Simulations}
We were also interested in the performance of the strategies on a diadic rectangles and on a given rectangle whose decomposition intro Haar-functions has high coefficients. These show the validity of our estimates for such rectangles, and verify the logic of the proof. For this purpose we chose the intervals $[0,\frac12]^d$ and
$[\frac13,\frac56]^d$, the first of which is diadic while the other has a very complex diadic decomposition. Comparison between those rectangles in one and two dimensions are given in Table~\ref{tab:dis_values_cases} and Figure~\ref{fig:discrepancy_cases}. The results clearly indicate the the biases of these rectangles are dominated by a different power of $\log(n)$.

\begin{table*}\centering
\ra{1.2}
\scalebox{0.95}{
\begin{tabular}{@{}rlllclll@{}}\toprule
& \multicolumn{3}{c}{$d = 1$} & \phantom{}& \multicolumn{3}{c}{$d = 2$} \\
\cmidrule{2-4} \cmidrule{6-8}
& \footnotesize{\text{Monte Carlo}} & \hspace{9pt}\footnotesize{\text{Haar}} & \footnotesize{\text{greedy-Haar}} && \footnotesize{\text{Monte Carlo}} & \hspace{9pt}\footnotesize{\text{Haar}} & \footnotesize{\text{greedy-Haar}}\\ \midrule
$R=[0, 1/2)^d$\\
$n=10$ 		& 1.0 (1.0) 	& 1.2 (0.9) 	& 1.1 (0.7) 	&& 1.1 (0.6) 	& 1.6 (1.2)		& 0.8 (0.5)	 \\
$n=100$ 	& 4.9 (3.4) 	& 5.3 (3.0) 	& 1.1 (1.5) 	&& 3.2 (2.2)	& 3.8 (2.7)		& 1.9 (1.7)  \\
$n=1000$ 	& 10.1 (7.0)	& 7.9 (6.2) 	& 2.0 (2.2) 	&& 8.2 (6.1)	& 7.5 (4.3) 	& 4.7 (3.4)  \\
$n=10000$	& 42.5 (36.2)	& 22.9 (18.3) 	& 2.5 (2.7) 	&& 22.4 (17.9)	& 27.1 (17.9)	& 5.1 (4.5) \\
$n=100000$ 	& 102.6 (57.3)	& 28.2 (26.7)	& 3 (2.2)	&& 76.2 (54.8)	& 73.1 (75.0)	& 6.7 (4.9) \\
$R=[1/3, 5/6)^d$\\
$n=10$ 		& 1.2 (0.8)	 & 1.4 (0.9)	 & 0.9 (0.6)	&& 1.1 (0.7)	& 1.1 (.06) 	& 1.2 (0.7)  \\
$n=100$ 	& 3.4 (2.2)	 & 3.1 (2.6)	 & 2.8 (1.7) 	&& 3.1 (2.2)	& 4.1	(3.0)	& 3.0 (1.7)  \\
$n=1000$	& 10.5 (8.5)	 & 9.1 (10.1)	 & 5.4 (4.0)	&& 10.2 (6.7)	& 12.9 (7.4)	& 5.0 (3.1)  \\
$n=10000$	& 36.5 (25.0)	 & 22.9 (17.8)	 & 6.5 (5.7)	&& 36.1 (22.3)	& 40.6 (24.8)	& 13.8 (12.5) \\
$n=100000$ 	& 123.9 (118.1)	 & 57.3 (37.7)	 & 10.4 (6.2)	&& 102.3 (87.4)	& 131.4 (88.7)	& 30.6 (22.1) \\
\\
\bottomrule
\end{tabular}}
\caption{Mean (standard error) biases for different rectangles in one and two dimensions.}
\label{tab:dis_values_cases}
\end{table*}
\begin{figure}
  \centering
    \begin{subfigure}{.45\linewidth}
    \centering
    \includegraphics[width=1.\linewidth]{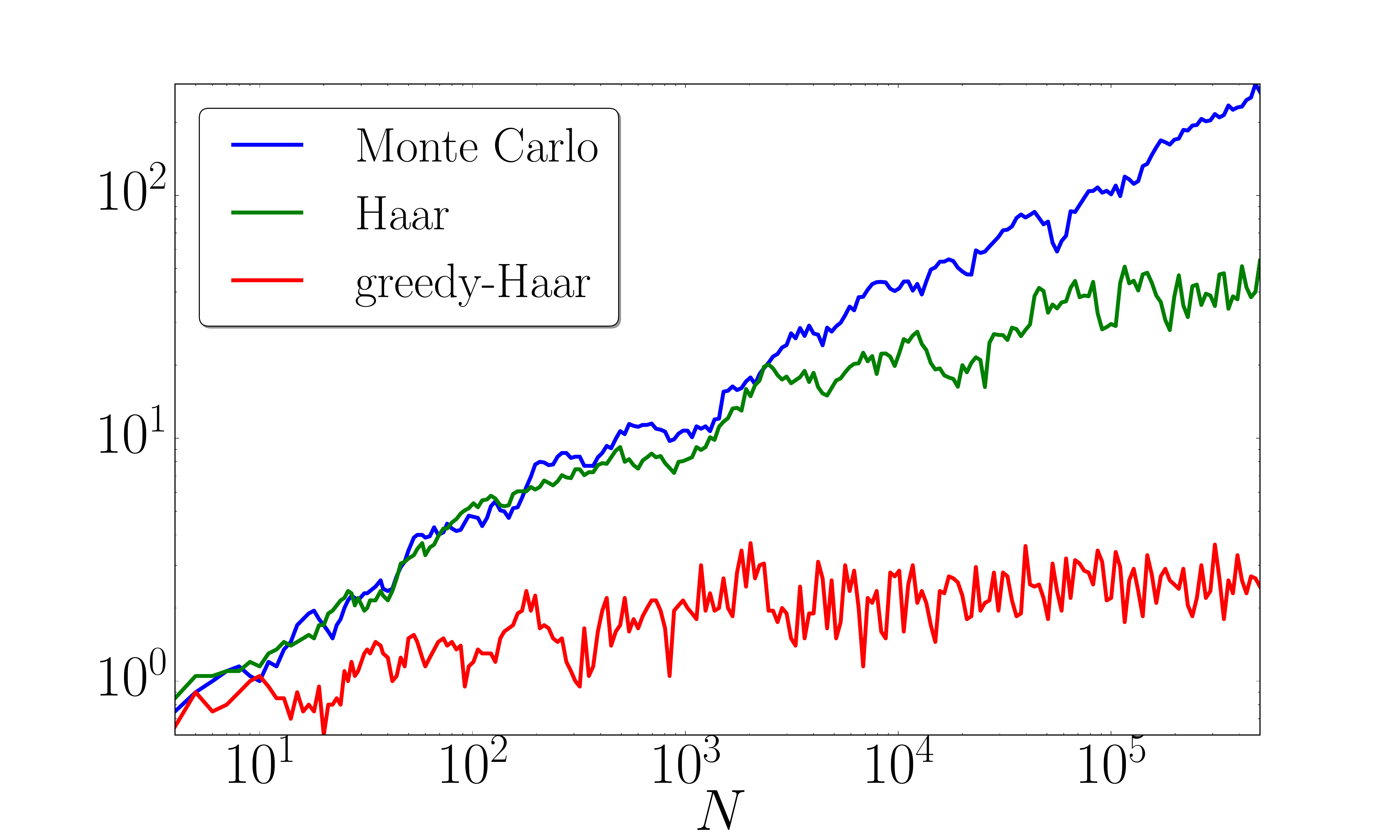}
    \caption{$R=[0, 1/2)$}
  \end{subfigure}
    \begin{subfigure}{.45\linewidth}
    \centering
    \includegraphics[width=1.\linewidth]{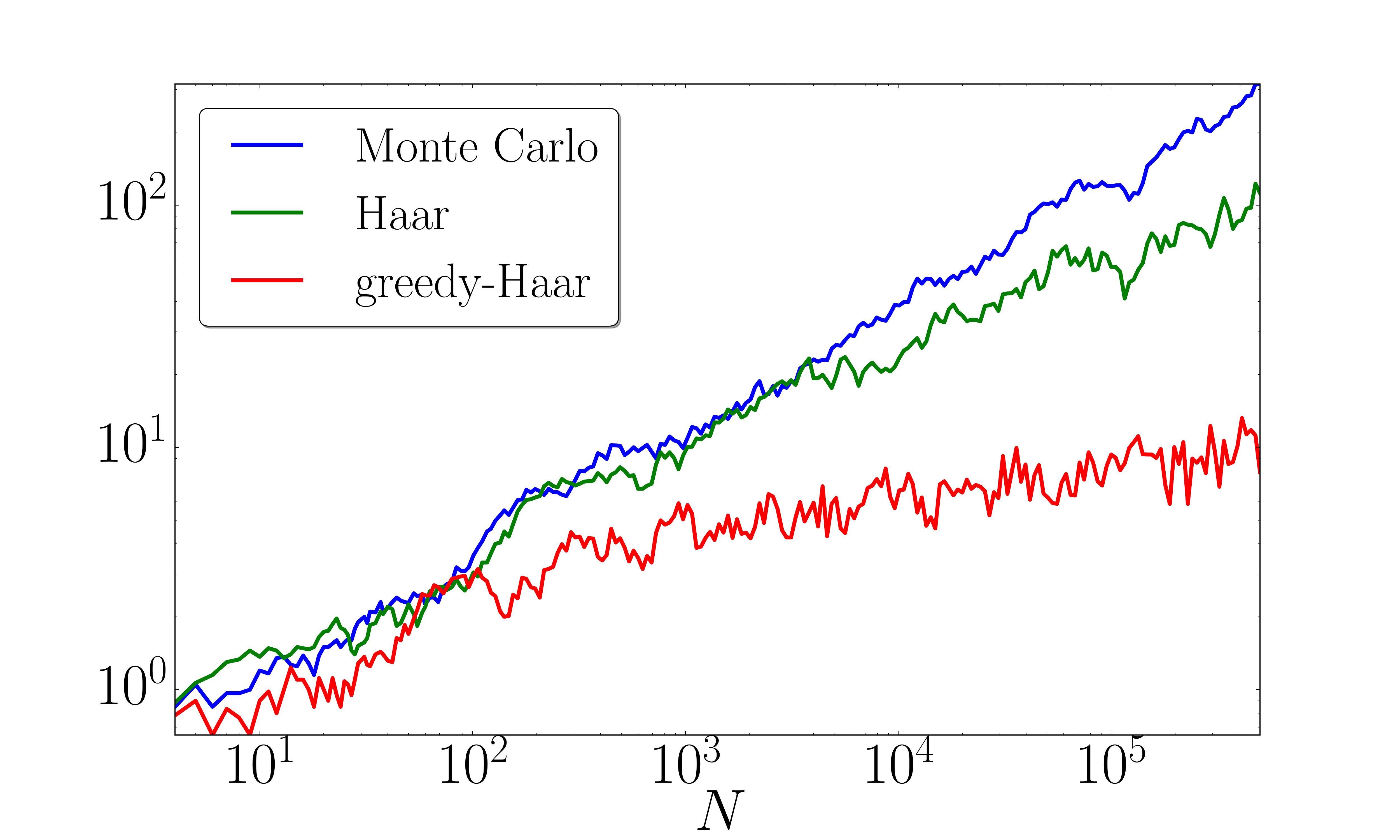}
    \caption{$R=[1/3, 5/6)$}
  \end{subfigure}

    \begin{subfigure}{.45\linewidth}
    \centering
    \includegraphics[width=1.\linewidth]{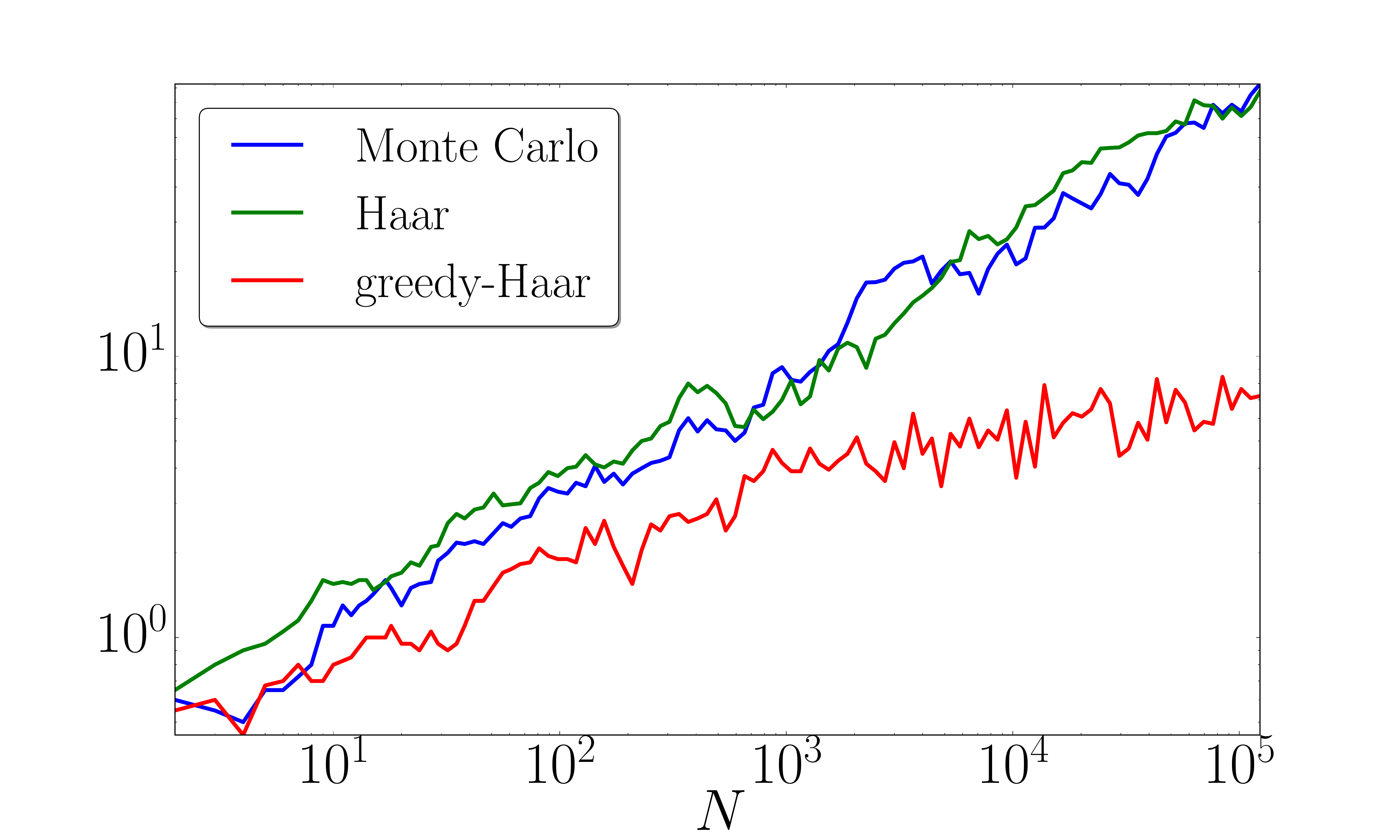}
    \caption{$R=[0, 1/2)^2$}
  \end{subfigure}\vspace{3mm}
    \begin{subfigure}{.45\linewidth}
    \centering
    \includegraphics[width=1.\linewidth]{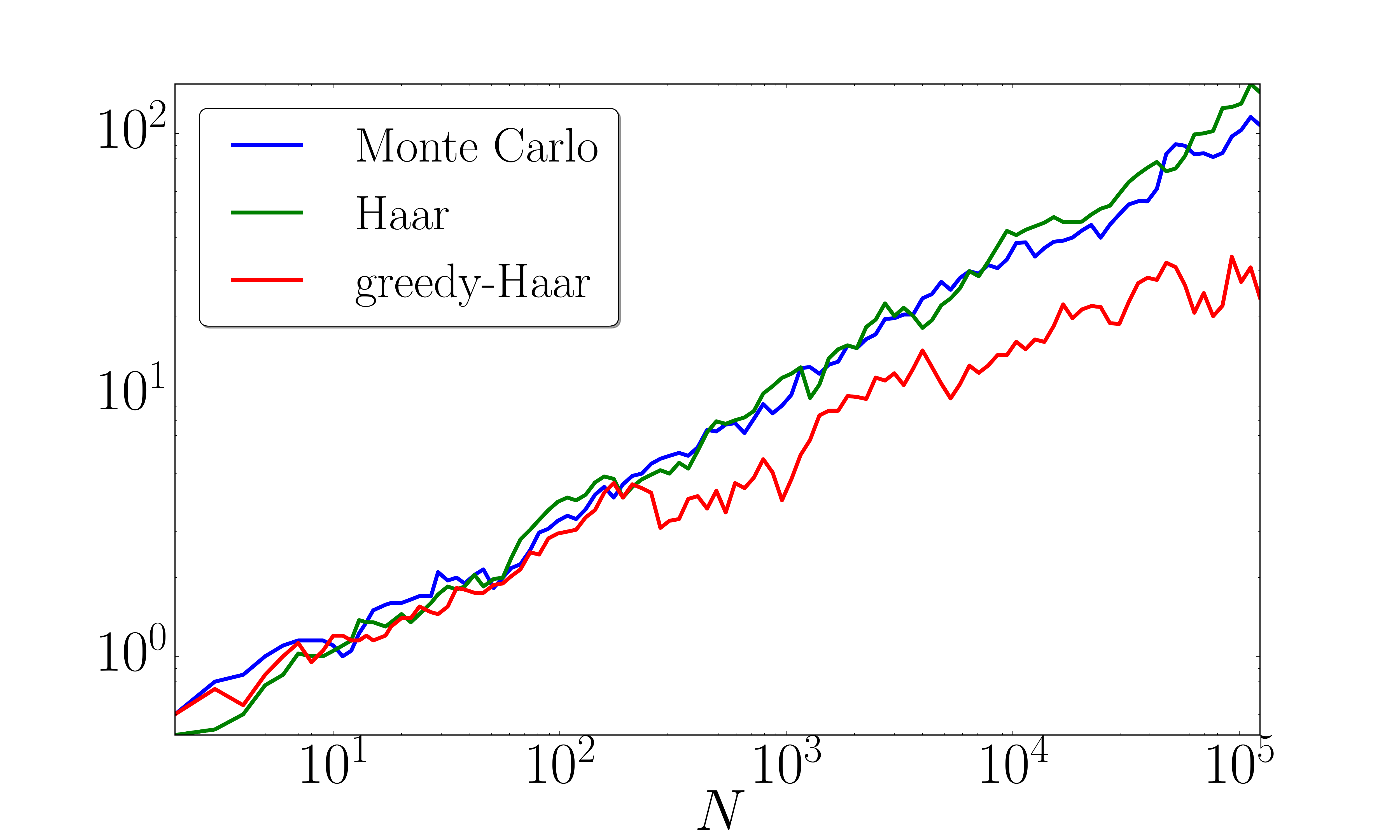}
    \caption{$R=[1/3, 5/6)^2$}
  \end{subfigure}\vspace{3mm}

  \caption{Plots of the bias $|\nu(R)-|R||$ in logarithmic scale averaged over $20$ experiments for different rectangles $R$ in one and two dimensions.
 }
 \label{fig:discrepancy_cases}

\end{figure}
\subsection{Observations from the simulations}
We draw the following observations from the simulations
\begin{itemize}
\item Both Haar and greedy-Haar seem to be always at least as good as Monte-Carlo sampling.
\item Greedy-Haar strategy seem to be always at least as good the Haar strategy.
\item In one dimension greedy-Haar performs significantly better than Monte-Carlo sampling for as little as 50 samples.
%\item In two and three dimension greedy haar performs significantly better than Monte-Carlo sampling for \ofcomment{HOWMANY?} samples.
\end{itemize}

\section*{Acknowledgments}
The authors wish to thank Itai Benjamini for suggesting the model of the power of two choices
on interval partitions, to Yuval Peres for introducing to us related recent litrature, and to Art Owen and Asaf Nachmias for useful discussions.

%\ofcomment{Commented out bib and app experiments}

\end{document}